\documentclass[10pt, singlespacing, twocolumn]{article}

\usepackage{amsmath}
\usepackage{amssymb}
\usepackage{amsthm}
\usepackage{graphicx}
\usepackage{float}
\PassOptionsToPackage{hyphens}{url}\usepackage{url}

\usepackage{algorithm}
\usepackage{algpseudocode}
\usepackage{stackengine}
\usepackage{subcaption}

\usepackage{stfloats}

\newtheorem{thm}{Theorem}

\newtheorem{prop}[thm]{Proposition}

\newcommand{\mbb}{\mathbb}
\renewcommand{\Re}{\mbb{R}}

\newcommand{\set}[2]{\left\{ #1\ \left| \ #2 \right. \right\}}

\newcommand{\innerprod}[2]{\langle{#1},{#2}\rangle}

\usepackage{bbm}
\usepackage{enumerate}
\usepackage{flushend}

\usepackage{array}
\newcolumntype{C}[1]{>{\centering\arraybackslash}m{#1}}
\usepackage{multirow}

\usepackage[disable]{todonotes}
\usepackage[top=2cm, bottom=1.5cm, left=1.5cm, right=2cm]{geometry}

\usepackage{epstopdf}
\usepackage{setspace}
\allowdisplaybreaks

\newcommand\mystretch{1.05}
\newcommand\mywidth{0.95}

\usepackage{sectsty}

\sectionfont{\fontsize{10}{10}\selectfont}
\subsectionfont{\normalfont\itshape\fontsize{10}{10}\selectfont}
\subsubsectionfont{\normalfont\itshape\fontsize{10}{10}\selectfont}

\renewcommand \thesection{\Roman{section}}
\renewcommand{\thesubsection}{\Alph{subsection}}
\renewcommand{\thesubsubsection}{\arabic{subsubsection})}
\usepackage[title]{appendix}

\begin{document}
\title{\vspace{-1cm}Approximate Dynamic Programming for Delivery Time Slot Pricing: \\a Sensitivity Analysis}
\author{Denis Lebedev, Kostas Margellos, 
Paul Goulart 
\thanks{Research is supported by SIA Food Union Management. Corresponding author: Denis Lebedev.}
\thanks{The authors are with the Department of Engineering Science, University of Oxford, Oxford, OX1 3PJ, United Kingdom (e-mail: \{denis.lebedev, kostas.margellos, paul.goulart\}@eng.ox.ac.uk).}}
\date{}
\maketitle


\textbf{\textit{Abstract}---We consider the revenue management problem of finding profit-maximising prices for delivery time slots in the context of attended home delivery. This multi-stage optimal control problem admits a dynamic programming formulation that is intractable for realistic problem sizes due to the so-called ``curse of dimensionality''. Therefore, we study three approximate dynamic programming algorithms both from a control-theoretical perspective and in a parametric numerical case study. Our numerical analysis is based on real-world data, from which we generate multiple scenarios to stress-test the robustness of the pricing policies to errors in model parameter estimates. Our theoretical analysis and numerical benchmark tests show that one of these algorithms, namely gradient-bounded dynamic programming, dominates the others with respect to computation time and profit-generation capabilities of the delivery slot pricing policies that it generates. Finally, we show that uncertainty in the estimates of the model parameters further increases the profit-generation dominance of this approach.}


\section{Introduction}\label{sec:intro}
Online grocery sales have been on the rise for the past few years. U.S. households are predicted to spend over \$100 billion per year on online grocery shopping in 2022 according to the Food Marketing Institute \cite{FMI2018} and up to \$133.8 billion per year according to GlobalData \cite{ONESPACE2018}. Similar developments take place in Europe where, e.g. in the Netherlands, online grocery sales are predicted to grow annually by 14\% to reach a share of 16\% of all grocery sales in 2030 \cite{ROLANDBERGER2019}. 

One of the main holdbacks for growth of online supermarkets is the increased cost of home delivery compared with the logistics of brick-and-mortar supermarkets \cite{GALANTE2013}. Moreover, another logistical problem for online supermarkets is that they have to fulfil \emph{attended} home delivery, i.e.\ to deliver groceries to customers in pre-agreed delivery time windows. To this end, customers are asked to select a delivery time window as part of the purchase on the sales website. From the side of the product-selling company, this poses an optimisation question: How should one optimally adjust prices for delivery slots over time to maximise profits, by taking into account how customers respond to price changes and how customer choice affects delivery costs? We denote this question as the \emph{revenue management problem in attended home delivery}.

Broadly speaking and independently of any specific problem characteristics, revenue management problems in attended home delivery can be viewed as optimal control problems. The dynamics of customers choosing delivery time windows on the booking website form the \emph{plant} that we seek to control. We can measure the customer choice behaviour by keeping track of placed orders, which we can use as \emph{state feedback}. An optimal control law would then use information from the states to update delivery slot prices, which serve as \emph{control inputs} to our plant as shown schematically in Fig.\ \ref{fig:controlview} below.

\begin{figure}[H]
\centering
\includegraphics[width=0.6\linewidth]{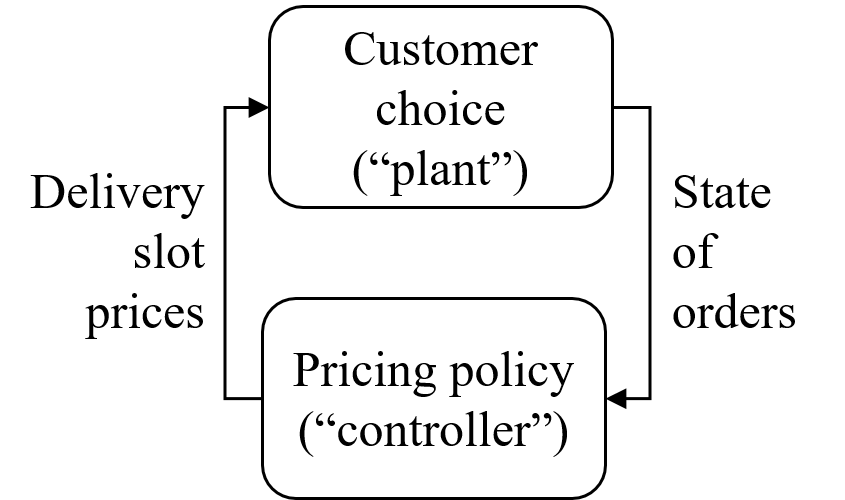}
\caption{A feedback control view of the revenue management in attended home delivery problem.}
\label{fig:controlview}
\end{figure}

In principle, the exact state of orders is high-dimensional. For example, it would need to represent locations of all customers and their chosen delivery time slot, if any. For industry-sized problems the number of states required to compute the optimal pricing policy exactly for all states becomes prohibitively large. Therefore, several model simplifications have been proposed in the literature; see \cite{STRAUSSETAL2018} for an overview.

In this paper, we focus on the state-space representation of \cite{YANG2017,YANG2016}. In this set-up, we split the delivery area  into several sub-areas, each of which is served by a single delivery vehicle. We can then solve the optimal delivery slot pricing problem for each delivery sub-area separately. In that case, the dimensionality of the state is the number of delivery time slots of any delivery day. The state of orders can then be thought of as a vector whose entries represent the number of deliveries for every delivery time slot in a particular sub-area. 
 
For this problem, \cite{YANG2017} proposes a dynamic programming (DP) formulation and an approximate DP scheme. Their algorithm approximates the exact value function of the DP as an affine function in the vector of states. In its conclusions, \cite{YANG2017} suggests that a possible direction for further research would be to explore non-linear approximate value functions. This direction is also proposed for future exploration in the conclusions of \cite{KOCH2020}. The recently developed approximate DP algorithm in \cite{ZHANG2019} was derived outside the attended home delivery literature and it provides such non-linear approximate value functions. However, we show in this paper that both of these algorithms suffer from theoretical limitations from an optimal control view. These can be overcome by a third approximate DP algorithm, presented in \cite{LEBEDEVETAL2020A}, termed gradient-bounded DP. The main contributions of this paper are thus:
\begin{enumerate}
\item We explain the theoretical limitations of the affine value function approximation algorithm from \cite{YANG2017} by showing that it results in an open loop controller, thus motivating the use of non-linear value function approximations, which provide state feedback.
\item We adapt the non-linear stochastic dual DP algorithm from \cite{ZHANG2019} to the revenue management problem in attended home delivery and show we can use a simplified version of the algorithm, which neglects a regularisation term that is redundant for the problem under study in this paper. However, despite it being able to work as a closed loop controller, we show that non-linear stochastic dual DP imposes difficulties in the computation of the optimal control policy since it requires the solution of a non-convex optimisation problem.
\item We then show that our gradient-bounded DP algorithm from \cite{LEBEDEVETAL2020A} can overcome the limitations of the other two algorithms, since it provides non-linear value function approximations, which can be computed using convex optimisation. We demonstrate the efficacy of this approach by means of a multi-scenario case study based on data from \cite{YANG2017}, in which we benchmark the performance of all three algorithms when the model parameters are known exactly. Furthermore, we include a detailed model parameter sensitivity analysis, which stress-tests the robustness of the algorithms against modelling errors, thus seeking to provide a balanced view on how these algorithms are likely to perform in practice. Such a comparative numerical case study is new for the problem as formulated in the next section and hence, it complements the numerical studies on attended home delivery conducted on different problem formulations, e.g.\ in \cite{KOCH2020}.
\end{enumerate}

The structure of the remaining paper is as follows: We formally define the underlying optimisation problem and its DP formulation in Section \ref{sec:problem}. Section \ref{sec:gen_approxalg} presents a general, sample-based approximate DP algorithm and we explain how the three algorithms considered in this paper are special cases of this general algorithm. In that section, we also elaborate on the main theoretical limitations of the first two algorithms and how the third overcomes these. Since the profits generated by all three algorithms are random variables, we explain how we quantify their profit-generation performance in Section \ref{sec:bounds_case}. We then demonstrate how the gradient-bounded DP algorithm dominates the two other algorithms by means of a multi-scenario numerical case study and robustness to uncertainty analysis in Section \ref{sec:casestudy}. Finally, we conclude the paper and provide some directions for future research in Section \ref{sec:conclusions}.

\emph{Notation:} Given some $s\in \mathbb{N}$, let $1_s$ be a column vector of all zeros apart from the $s$-th entry, which equals 1. Furthermore, we define the convention that $1_0$ is a vector of zeros, i.e.\ $1_0:=0$. Let $\mathbf{1}$ denote a vector of ones. Let $\innerprod{\cdot}{\cdot}$ denote the standard inner product of its arguments. Let $\mathbb{E}$ denote the expectation operator, let $\Pr(\cdot)$ denote the probability of its argument and let $\mathbbm{1}(\cdot)$ denote the indicator function.

\section{Problem statement}\label{sec:problem}
\subsection{Multi-stage optimal control problem formulation}
Revenue management in attended home delivery can be formulated as the following multi-stage optimal control problem for any delivery sub-area served by a single delivery vehicle: Customers are assumed to be allowed to make bookings in a finite time horizon and there is only a finite number of times that the online vendor can change delivery slot prices. Therefore, we consider a finite and discrete time horizon $T:=\{1,2,\dots,\bar{t}\}$. There is an additional time step $\bar{t}+1$, at which no bookings happen anymore, which we will use to define the terminal condition of the problem. Suppose that the delivery day is split into $n$ delivery time slots. Denote the set of delivery time slots by $S:=\{1,2,...,n\}$. As mentioned in Section \ref{sec:intro}, we focus on an aggregated state-space representation, where for any time step $t\in T\cup\{\bar{t}+1\}$ we define a state vector $x_t \in X \subset \mathbb{Z}^n$, whose entries are the number of orders placed in the respective delivery time slots. The set $X$ is defined by the maximum state vector $\bar{x}$, i.e.\ $X:=\set{x_t\in \mathbb{Z}^n}{0\leq x_t \leq \bar{x}}$. For any $t\in T$, we define the delivery charge vector \mbox{$d_t:=[d_{1,t},d_{2,t},\dots,d_{n,t}]^{\intercal}$}. Let the set of admissible delivery charge vectors be \mbox{$D:=\set{d_t}{d_{s,t} \in \left[\underline{d},\bar{d}\,\right]\cup\{\infty\} \text{ for all } s \in S}$}.

For any $s\in S$, define the transition probability between two states $x_t$ and $x_{t+1}=x_t+1_s$ under delivery price vector $d_t$ as $P_s(d_t)$, where we require $P_s(d_t)\geq 0$ for all $(s,d_t)\in S\times D$. We impose that $\sum_{s\in S}P_s(d_t)<1$, such that the probability of the customer not choosing \emph{any} slot is defined as $P_0(d_t)=1-\sum_{s\in S} P_s(d_t)$. This requirement implies that transitions from $x_t$ to $x_{t+1}$ are only possible in the positive direction and by at most a unit step along one dimension. Such models are typical for order-taking processes (see \cite{ASDEMIRETAL2009,SUH2011,YANG2017,YANG2016}). For the purpose of the case study, we will assume that the customer choice model follows a multinomial logit model, like in \cite{DONGETAL2009,YANG2017,YANG2016}, i.e.\
\begin{equation}\label{eq:mnl}
P_s(d_t):= \frac{\exp(\beta_c+\beta_s+\beta_d d_{s,t})}{\sum_{k\in S}\exp(\beta_c+\beta_k+\beta_dd_{k,t})+1},
\end{equation}
for all $(s,d)\in S\times D$, where $\beta_c \in \Re$ denotes a constant offset, $\beta_s \in \Re$ represents a measure of the popularity for all delivery slots and $\beta_d < 0$ is a parameter for the price sensitivity. Note that the no-purchase utility is normalised to zero, i.e.\ for the no-purchase ``slot'' $s=0$, we have a no-delivery ``charge'' $d_{0,t}=0$, such that \mbox{$\beta_c+\beta_0+\beta_d d_{0,t} = \beta_c+\beta_0 = 0$} and hence, the $1$ in the denominator of \eqref{eq:mnl} arises from $\exp(\beta_c+\beta_0)=1$. Furthermore, note that the constant offset $\beta_c$ is not necessary, since it can be absorbed in the $\{\beta_s\}_{s\in S\cup \{0\}}$ parameters. However, $\beta_c$ is often kept in practice to normalise one of the $\{\beta_s\}_{s\in S\cup \{0\}}$ parameters to zero (see e.g.\ \cite{YANG2017}). Finally, we define an average revenue per order $r\in \Re$, an expected customer arrival rate (on the booking system) per time step $\lambda \in (0,1]$ and an approximate delivery cost function $C: \mathbb{Z}^n \to \Re\cup \infty$, which we assume is Lipschitz continuous. The role of infinite delivery costs is to indicate infeasible states. We construct a multi-stage optimal control problem of the following form:
\begin{align}\label{eq:multistage}
&\underset{\{d_t\in D\}_{t=1}^{\bar{t}}}{\max} -C(x_{\bar{t}+1})+\sum_{t\in T} \innerprod{x_{t+1}-x_t}{d_t+r},\nonumber \\
&\text{subject to } x_{t+1} = x_t + \xi_t, \text{ for all } t\in T,
\end{align}
where $\xi_t=1_s$ with probability $\lambda P_s(d_t)$ and $\xi_t=0$ with probability $1-\lambda\sum_{s\in S}P_s(d_t)$, if $x_t+1_s \in X$ for all $s\in S$. In the opposite case when $x_t+1_s \notin X$, we have $\xi_t=0$ for all $s\in S$, i.e.\ we do not allow orders that increase the state $x$ beyond the feasible set $X$. From an economic perspective, the objective value is the total expected operational contribution margin, i.e.\ revenue from sales and delivery charges minus delivery costs. For simplicity, we will refer to this objective as simply the \emph{expected profit} that we seek to maximise.

\subsection{Dynamic programming formulation}
The above multi-stage optimal control problem is stage-wise independent, which makes it possible to derive the following DP recursion, analogously to \cite{YANG2017}, by introducing the value function $V_t: \mathbb{Z}^n \to \Re \cup -\infty$, which represents the expected profit-to-go for any state-time pair $(x,t)\in X\times T$ as follows:
\begin{align}\label{eq:dp}
V_t(x) :=&\; \underset{d\in D}{\max}\left\{\lambda\sum_{s\in S}P_s(d)\left(r+d_s+V_{t+1}(x+1_s)\right)\right.\nonumber\\
&\;\left.+\left(1-\lambda\sum_{s\in S}P_s(d)\right)V_{t+1}(x)\right\}\nonumber\\
&\;\forall (x,t)\in X\times T,\text{ where } \nonumber\\
&\; V_{\bar{t}+1}(x)=-C(x) \quad \forall x\in X.
\end{align}
We assume that $V_t(x)=-\infty$ for all infeasible states $x\notin X$. Notice that we have dropped subscripts $t$ for $x$ and $d$ to simplify notation and since the time step is evident from the value function. Furthermore, we adopt the convention that when for any $s\in S$, it happens that $d_s=\infty$ and $V_{t+1}(x+1_s)=-\infty$, we have $P_s(d)(r+d_s+V_{t+1}(x+1_s))=0$. This corresponds to the additional profit of accepting an unavailable slot, which is mathematically undefined in \eqref{eq:dp}, but practically it is zero. To represent the DP in a more compact form, we define the Bellman operator $\mathcal{T}$ through the relationship
\begin{equation}\label{eq:dpcompact}
V_t = \mathcal{T}V_{t+1}, \text{ for all } t\in T.
\end{equation}
Since for realistic problem instances, the above DP formulation in \eqref{eq:dp} cannot be solved by direct computation of the exact value function due to a prohibitively large number of states, one needs to approximate the value function of the DP. However, given an approximate value function, finding optimal prices is relatively easy. For example, for the multinomial logit model, \cite{DONGETAL2009} determines approximately optimal prices using a simple Newton root search.

\section{General, sample-based, iterative approximation algorithm}\label{sec:gen_approxalg}

As mentioned in the previous section, a key to solving the revenue management problem in attended home delivery is to approximate the value function of the DP in \eqref{eq:dp} effectively. A popular strategy, not only for this problem (see \cite{KOCH2020,YANG2017}), but also for other stochastic multi-stage problems (see \cite{PEREIRA1991,SHAPIRO2011}), is to use a sample-based approach and to refine the value function along states that are likely to be visited under the approximately optimal decision policy. Approximations can then be improved by iterating between generating samples and refining the value function along the obtained sample paths.

We now state a general, sample-based, iterative approximate DP procedure in Algorithm \ref{alg:gen_approx} below. The three algorithms that we investigate in this paper are special cases of this general algorithm and they differ only in step 11 of Algorithm \ref{alg:gen_approx}. We detail how this step is computed for each of the individual specific algorithms further below.

\begin{algorithm}[h]
\setstretch{\mystretch}
\caption{General sample-based, iterative approximation algorithm}\label{alg:gen_approx}
\begin{algorithmic}[1]
\State Initialise parameters: $X, D, P_s, T, r, \lambda, C$ and $i_{\max}$
\State Initialise $Q_t^{0}(x) \gets (\bar{d}+r)\innerprod{\mathbf{1}}{\bar{x}-x}-C(\bar{x})$, for all $(x,t)\in X\times T$
\State Initialise $Q_{\bar{t}+1}^0(x) \gets -C(x)$, for all $x \in X$
\For{$i\in \{1,2,\dots, i_{\max}\}$}
	\State $x_1^i \gets \mathbf{0}$
	\For{$t\in T$} \Comment ``Forward sweep''
		\State $d_t^i \gets d^*(x_t^i)$, the solution of \eqref{eq:forward_prob}
		\State $x_{t+1}^i \gets x_{t}^i + \underset{x_{t+1}^i}{\text{sample}}\left\{P_{s}\left(d_t^i\right)\right\}$
	\EndFor
	\While{$t > 1$} \Comment ``Backward sweep''
  		\State $Q_t^{i}\gets $ update $Q_t^{i-1}$
  		\State $t \gets t-1 $ 
	\EndWhile
\EndFor
\end{algorithmic}
\end{algorithm}

We first initialise all parameters of the DP in \eqref{eq:dp} (step 1). Denote the maximum number of iterations by $i_{\max} \in \mathbb{N}$ and let $I:=\{0,1,\dots,i_{\max}\}$. Let the value function approximation be denoted by $Q_t^{i}$ for all $(i,t)\in I\times T$. We could initialise $Q_t^0$ to any value as long as that does not violate any assumptions on the approximation algorithm used, as discussed further below. However, one effective way to satisfy the assumptions of all three algorithms considered and to speed up computation, is to initialise $Q_{t}^0$ for all $t\in T$ using the unique fixed point of DP, $V^*$ (step 2). In \cite{LEBEDEVETAL2019A}, it is shown that the fixed point is given analytically (under mild technical assumptions) as
\begin{equation}\label{eq:fixedpoint}
V^*(x):= (\bar{d}+r)\innerprod{\mathbf{1}}{\bar{x}-x}-C(\bar{x}),\text{ for all } x\in X.
\end{equation}
Note that $V^*$ is an upper bound to $V_t$ for all $t\in T$, since $\mathcal{T}$ is a monotone operator (see \cite[Chapter 3]{BERTSEKAS2012}). Furthermore, notice that the fixed point is affine in $x$ and that the components of the gradient are given by $\bar{d}+r$. Using the Bellman equation in \eqref{eq:dp}, we can use this gradient to compute the optimal delivery slot prices at the fixed point. This leads to the optimal delivery charge being $\bar{d}$ for all feasible time slots. The intuition behind this is that the fixed point corresponds to the limit of the value function as $t$ tends to $-\infty$. Hence, going backwards infinitely many time steps, the probability of selling out the entire delivery capacity across all delivery slots tends to one for all prices $d\in D$. The profit-maximising behaviour in this hypothetical scenario would then be to charge customers the maximum admissible delivery charge $\bar{d}$ for all delivery time slots. Finally, we also initialise $Q_{\bar{t}+1}^i(x) := V_{\bar{t}+1}(x) = - C(x)$ for all $(x,i)\in X\times I$ (step 3).

For the next steps, fix any iteration $i\in I\setminus\{0\}$. In each ``forward sweep'', we solve an approximate version of the Bellman equation in \eqref{eq:dp} forward in time, i.e.\ by replacing $V_t$ with its approximation $Q_t^{i-1}$ (step 7), i.e.\
\begin{align}\label{eq:forward_prob}
d^*(x_t^i) :=&\; \underset{d\in D}{\mathrm{argmax}}\left\{\lambda\sum_{s\in S}P_s(d)\left(r+d_s+Q_{t+1}^{i-1}(x_t^i+1_s)\right)\right.\nonumber\\
&\;\left.+\left(1-\lambda\sum_{s\in S}P_s(d)\right)Q_{t+1}^{i-1}(x_t^i)\right\}.
\end{align}
Notice that \cite[Theorem 1]{DONGETAL2009} have shown that for the multinomial choice model, the maximisers are unique for the above expression, hence we use equality in \eqref{eq:forward_prob}. By \eqref{eq:forward_prob}, we compute a suboptimal $d_t^i$ for all $t \in T$ and simulate state transitions by sampling from the transition probability distribution given the  approximately optimal decisions (step 8). This defines a sample path $x_{t}^i $ for all $t \in T\cup \{\bar{t}+1\}$. 

In each ``backward sweep'', we update the value function approximation using the particular mechanism of the chosen algorithm (step 11).  These two sequences -- ``forward sweep'' and ``backward sweep'' -- are repeated for $i_{\max}$ iterations. In the next sections, we describe the exact mechanisms of the three algorithms considered in this paper, making step 11 in Algorithm \ref{alg:gen_approx} explicit.

\subsection{Affine value function approximation update}\label{sec:avfa_update}
This approach is proposed in \cite{YANG2017}. The idea is to approximate the value function by an affine function of the form
\begin{equation}\label{eq:avfa}
Q_t^i(x) := \gamma_0^i +(\bar{t}+1-t)\theta^i - \sum_s \gamma_s^i x_s,
\end{equation}
for all $(x,i,t)\in X\times I\times T$ and where $\gamma_s^i$, for all $s\in S\cup \{0\}$ and $\theta^i$ are scalar, real-valued parameters, for all $i\in I$. Then, the updating rule in step 11 of Algorithm \ref{alg:gen_approx} is a gradient descent step to minimise $\left(Q_t^{i-1}(x_{t+1}^i)-\mathcal{T}Q_{t+1}^{i-1}(x_{t+1}^i)\right)^2$, which thus becomes
\begin{align}
\gamma_0^i =&\; \gamma_0^{i-1} -\alpha_1 \left(Q_t^{i-1}(x_{t+1}^i)-\mathcal{T}Q_{t+1}^{i-1}(x_{t+1}^i)\right)\nonumber\\
\gamma_s^i =&\; \gamma_s^{i-1} -\alpha_2 \left(Q_t^{i-1}(x_{t+1}^i)-\mathcal{T}Q_{t+1}^{i-1}(x_{t+1}^i)\right)x_{s,t+1}^i,\nonumber\\
&\;\text{for all } s\in S\nonumber\\
\theta^i =&\; \theta^{i-1} - \alpha_3
\left(Q_t^{i-1}(x_{t+1}^i)-\mathcal{T}Q_{t+1}^{i-1}(x_{t+1}^i)\right)(\bar{t}+1-t)\label{eq:avfa_update},
\end{align}
where $\alpha_1$, $\alpha_2$ and $\alpha_3$ are (positive) step sizes, which are chosen to be sufficiently small for convergence of the above iterative procedure (see e.g. \cite[Lemma 8.2]{BECK2017}).

One important observation from a control perspective is that a value function approximation that is affine in $x$ implies that the pricing control will have \emph{no state feedback} for all states $x$ such that $x+1_s\in X$ for all $s\in S$. To see this, notice that \eqref{eq:forward_prob} can be re-written in terms of differences $Q_{t+1}^{i-1}(x_t^i)-Q_{t+1}^{i-1}(x_t^i+1_s)$ for all $s\in S$ as follows:

\begin{align}\label{eq:forward_prob2}
d^*(x_t^i) =&\; \underset{d\in D}{\mathrm{argmax}}\left\{\lambda\sum_{s\in S}P_s(d)\left(r+d_s\right.\right.\nonumber\\
&\left.\vphantom{\sum_s}\left. +\;Q_{t+1}^{i-1}(x_t^i+1_s)-Q_{t+1}^{i-1}(x_t^i)\right)+Q_{t+1}^{i-1}(x_t^i)\right\} \nonumber\\
=&\;\underset{d\in D}{\mathrm{argmax}}\left\{\lambda\sum_{s\in S}P_s(d)\left(r+d_s-\gamma_s^{i-1}\right)+Q_{t+1}^{i-1}(x_t^i)\right\} \nonumber\\
=&\;\underset{d\in D}{\mathrm{argmax}}\left\{\lambda\sum_{s\in S}P_s(d)\left(r+d_s-\gamma_s^{i-1}\right)\right\},
\end{align}
for all $x\in X$, such that $x+1_s \in X$ for all $s\in S$, and where we have first substituted for $Q_{t+1}^{i-1}$ from \eqref{eq:avfa} and then cancelled the term that is independent of $d$ since it is irrelevant for the argmax operator. Hence, the approximately optimal pricing policy does not depends on the state $x_t^i$ for all $(x,t)\in X\times T$ such that $x+1_s\in X$ for all $s\in S$. This ultimately means that the affine value function approximation generates a \emph{feedforward} pricing policy, which is incapable of adjusting prices based on changes in the vector of orders. This insight also provides theoretical support for the suggestions of \cite{KOCH2020} and \cite{YANG2017} to explore non-linear value function approximations: The preceding discussion shows that allowing the value function to be non-linear makes it possible to include state feedback in the pricing policy.

\subsection{Non-linear stochastic dual dynamic programming update}
In contrast to the affine value function update above, the non-linear stochastic dual  DP update generates non-linear value function approximations, which by the above discussion make it possible to include state feedback in the pricing policy. Similarly to \cite{ZHANG2019} and \cite{ZOU2019}, this update is computed in step 11 of Algorithm \ref{alg:gen_approx} as
\begin{equation}
Q_{t}^i \gets \min\{H^*,Q_t^{i-1}\},
\end{equation} 
where the minimum is taken pointwise and the so-called Lagrange dual cut $H^*$ is defined as
\begin{subequations}\label{eq:lagrangecut}
\begin{align}
& H^*(x):= v^*-\innerprod{\mu^*}{x_{t+1}^i - x}, \text{ for all } x\in X\\
&\text{and where } \nonumber\\
& v^*:=\underset{\mu\in \mathcal{M}}{\min} \; \underset{d\in D,z\in X}{\max}\left\{\lambda\sum_{s\in S}P_s(d)\left(r+d_s+Q_{t+1}^{i-1}(z+1_s)\right)\right.\nonumber\\
& \hphantom{v^*:=}\;\left.+\left(1-\lambda\sum_{s\in S}P_s(d)\right)Q_{t+1}^{i-1}(z)+\innerprod{\mu}{x_{t+1}^i-z}\right\}\label{eq:vstar}
\end{align}
\end{subequations}
and where $\mu^*$ is the minimiser of \eqref{eq:vstar}. We now formally establish that the Lagrangian cut from \eqref{eq:lagrangecut} generates an upper bound on the exact value function. Notice that this means that the entire approximate value function $Q_t^ i$ is an upper bound on the exact value function $V_t$ if we additionally assume that the initialiser $Q_t^0$ is an upper bound. Using the unique fixed point of the DP for $Q_t^0$ at the initialising step, as discussed at the beginning of Section \ref{sec:gen_approxalg}, satisfies this assumption.
\begin{prop}\label{pr:lagrange}
Fix any $(i,t)\in (I\setminus\{0\})\times T$. The Lagrangian cut $H^*$, generated at a sample $x_{t+1}^{i}$, produces an upper bound on the exact value function, i.e.\
\begin{equation}
H^*(x) \geq V_t(x), \text{ for all } x\in X.
\end{equation}
\end{prop}
\begin{proof}
 Fix any $x_0 \in X$. Note that in the base case at $i=0$, we have $Q_t^i(x)\geq V_t(x)$ for all $(x,t)\in X\times T$, since $Q_t^0$ is initialised at the fixed point $V^*$, which is an upper bound to $V_t$. Suppose by means of an induction hypothesis that for some $(i,t)\in (I\setminus\{0\})\times T$, $Q_{t+1}^{i-1}(x)\geq V_{t+1}(x)$ for all $x\in X$. Fix any $x\in X$ and any compact set $\mathcal{M}\subset\Re^n$. Then we can upper bound the right-hand-side of the Bellman equation in \eqref{eq:dp} by introducing an auxiliary variable $z$ and by replacing $V_{t+1}(y)$ by $Q_{t+1}^{i-1}(y)$ for all $y\in \{z+1_s\}_{S\cup\{0\}}$, which yields
\begin{align}
V_t(x)\leq&\;\underset{d\in D,z\in X}{\max}\left\{\lambda\sum_{s\in S}P_s(d)\left(r+d_s+V_{t+1}(z+1_s)\right)\right.\nonumber\\
&\;\left.+\left(1-\lambda\sum_{s\in S}P_s(d)\right)V_{t+1}(z)+\innerprod{\mu}{x_{t+1}^i-z}\right\} \nonumber \\
&\;- \innerprod{\mu}{x_{t+1}^i-x}\nonumber \\
\leq&\;\underset{d\in D,z\in X}{\max}\left\{\lambda\sum_{s\in S}P_s(d)\left(r+d_s+Q_{t+1}^{i-1}(z+1_s)\right)\right.\nonumber\\
&\;\left.+\left(1-\lambda\sum_{s\in S}P_s(d)\right)Q_{t+1}^{i-1}(z)+\innerprod{\mu}{x_{t+1}^i-z}\right\} \nonumber \\
&\;- \innerprod{\mu}{x_{t+1}^i-x}\label{eq:nlsddpcut},
\end{align}
for all $\mu \in \mathcal{M}$. Note that the first inequality holds since we maximise over $z \in X$, but choosing $z=x \in X$, results in equality. To see this, notice that setting $z=x$ yields the original Bellman equation in \eqref{eq:dp} because the added and subtracted $\innerprod{\mu}{x_{t+1}^i-z}$ terms cancel. The second inequality in \eqref{eq:nlsddpcut} holds due to the induction hypothesis that $Q_{t+1}^{i-1}(x)\geq V_{t+1}(x)$ for all $x\in X$. Since the above expression holds for all $\mu \in \mathcal{M}$, we can minimise over $\mu\in \mathcal{M}$ to arrive at
\begin{align}
V_t(x)\leq&\; \underset{\mu\in \mathcal{M}}{\min} \; \underset{d\in D,z\in X}{\max}\left\{\lambda\sum_{s\in S}P_s(d)\left(r+d_s+Q_{t+1}^{i-1}(z+1_s)\right)\right.\nonumber\\
&\;\left.+\left(1-\lambda\sum_{s\in S}P_s(d)\right)Q_{t+1}^{i-1}(z)+\innerprod{\mu}{x_{t+1}^i-z}\right\} \nonumber\\
&\;- \innerprod{\mu^*}{x_{t+1}^i-x}\nonumber\\
& = v^*- \innerprod{\mu^*}{x_{t+1}^i-x} = H^*(x)\label{eq:hstar}
\end{align}
where $\mu^*\in \mathcal{M}$ denotes the minimiser of the maximised expression in curly brackets above, and where $v^*$ as well as $H^*$ are found from \eqref{eq:lagrangecut}, as required.
\end{proof}

The proof of this proposition is based on a similar result in \cite[Proposition 5]{ZHANG2019}, where an additional regularisation term is included in the approximate version of the forward problem in \eqref{eq:forward_prob}. The regularisation is needed for general multi-stage non-linear stochastic problems to guarantee that the approximate value function is Lipschitz continuous. For the specific revenue management problem under study, \cite{LEBEDEVETAL2019A} have shown that the DP can be re-written as a so-called stochastic shortest path problem (see \cite[Chapter 3]{BERTSEKAS2012}), where the Bellman operator $\mathcal{T}$ is known to be monotonic. Since the terminal condition of the DP and the fixed point are Lipschitz continuous functions, this implies that the value function is Lipschitz continuous for all time steps $t\in T\cup\{\bar{t}+1\}$ and that we do not need regularisation in the cut definition. Hence, our proof shows that the regularisation term can be dropped for the problem under study in this paper, which simplifies the computation of the Lagrange dual cut $H^*$.

From a control perspective, the benefit of having a non-linear value function approximation -- in comparison with the affine value function approximation from Section \ref{sec:gen_approxalg}-A -- comes at a different cost: The problem of finding the optimal cut coefficients $\mu$ in \eqref{eq:vstar} is a non-convex optimisation problem and there are consequently no guarantees that it can be solved to global optimality. For the particular form of the problem in this paper, we can find the cut coefficients from a reformulation of the problem, which results in a bi-concave objective function, which can be exploited to solve this problem as outlined in Appendix \ref{ap:biconcave_reform}. Nevertheless, since global optimality is required to ensure that the approximate value function constitutes an upper bound on the exact value function in accordance with Proposition \ref{pr:lagrange}, we cannot guarantee that a cut is indeed an upper bound on the exact value function. We illustrate how this may result in computational problems in Section \ref{sec:casestudy}.

\subsection{Gradient-bounded dynamic programming update}
Gradient-bounded DP was introduced in \cite{LEBEDEVETAL2020A} for the specific application of revenue management in attended home delivery. This manifests itself in the assumptions on the types of value function that the algorithm can approximate. In short, the exact value function of the DP needs to satisfy two assumptions: 

First, the value function needs to be \emph{concave extensible}. This means that its concave closure $\tilde{V}_t:\Re^n\to \Re$, defined as the smallest concave upper bound on the exact value function, coincides with the exact value function for all state-time pairs, i.e.\ $\tilde{V}_t(x)=V_t(x)$ for all $(x,t)\in X\times (T\cup\{\bar{t}+1\})$. Second, the exact value function needs to be \emph{submodular}. This is satisfied if and only if
\begin{equation}\label{eq:submod_case}
V_t(\max(y_1,y_2))+V_t(\min(y_1,y_2)) \leq V_t(y_1)+V_t(y_2),
\end{equation}
for all $(y_1,y_2,t)\in X\times X\times (T\cup\{\bar{t}+1\})$.

These two assumptions result in a particular segmentation of the convex hull of $V_t$: For any $t\in T$, construct the unique hyperplane $H$ through the set of pairs $(y,V_t(y))$ for all $y\in Y_+(x_{t+1}^i):=\{x_{t+1}^i+1_s\}_{s\in (S\cup\{0\})}$. Then $H$ is a separating hyperplane, i.e.\ $H(x)\geq V_t(x)$ for all $x\in X$, where the inequality holds with equality for all $y\in Y_+(x)$. The gradient-bounded DP algorithm exploits this property. We refer the interested reader to \cite{LEBEDEVETAL2020A} for details on the above-mentioned assumptions and to \cite{LEBEDEVETAL2019B} for proofs that these assumptions hold for the revenue management problem under study. 

For gradient-bounded DP, let the value function approximation $Q_t^{i}$ for all $(i,t)\in I\times T$ be the pointwise minimum of a finite number of affine functions, i.e.
\begin{equation}\label{eq:qminh}
Q_t^{i}(x):=\underset{j \in \{0,1,\dots,i\}}{\min}H_{t}^j(x), \text{ for all } x\in X,
\end{equation}
where $H_{t}^j: X \mapsto \Re$ describes a hyperplane, i.e.\ 
\begin{equation}
H_{t}^j(x):=\innerprod{a_{t}^j}{x} + b_{t}^j, \text{ for all } x\in X,
\end{equation}
with $a_{t}^j \in \Re^{n}, b_{t}^j \in \Re$ for all $(t,j)\in T\times I$. Furthermore, this approximation is an upper bound on the exact value function, i.e.\ $Q_t^i(x)\geq V_t(x)$ for all $(x,t,i)\in X\times T\times I$. To this end, it is important to initialise $Q_{t}^0$ for all $t\in T$ at an upper bound. The gradient-bounded DP update then ensures that the approximate value functions remain upper bounds to the exact value function for all iterations $i\in I$ as shown in \cite[Proposition 1]{LEBEDEVETAL2020A}. In step 11 of Algorithm \ref{alg:gen_approx}, gradient-bounded DP generates updates for the approximate value function as shown in Algorithm \ref{alg:gbdp_update} and explained further below.

\begin{algorithm}[H]
\setstretch{\mystretch}
\caption{Gradient-bounded dynamic programming update}\label{alg:gbdp_update}
\begin{algorithmic}[1]
\State $Z(x_{t+1}^i)\gets\{x_{t+1}^i+1_s+1_s'\}_{\begin{subarray}{l}
		 s\in S\cup\{\mathbf{0}\},\\
		s'\in S\cup\{\mathbf{0}\}\end{subarray}}$
\If{$Q_{t+1}^{i-1}$ is submodular on $Z(x_{t+1}^i)$}
	\State $H^* \gets$ unique hyperplane through \hphantom{wwwwwwwww} 
	$\hphantom{wwwwwwn}\vphantom{sum^1}\left\{\left(y,(\mathcal{T}Q_{t+1}^{i-1})(y)\right)\right\}_{y \in Y_{+}(x_{t+1}^i)}$
\Else
	\State $j^* \in \underset{j\in J_{t+1}^{i-1}}{\text{argmin}}\left\{\left(\mathcal{T}H_{t+1}^{j-1}\right)\left(x_{t+1}^i\right)\right\}$	
	\State $H^* \gets \mathcal{T}H_{t+1}^{j^*-1}$
\EndIf
\State $Q_t^{i}\gets \min\left\{H^*,Q_t^{i-1}\right\}$
\end{algorithmic}
\end{algorithm}

Fix any iteration $i\in I$. We first check if $Q_{t+1}^{i-1}$ is submodular (see \eqref{eq:submod_case}) on the set $Z(x_{t+1}^i):=\{x_{t+1}^i+1_s+1_{s'}\}$, for all $(s,s')\in (S\cup \{0\})\times (S\cup\{0\})$, i.e\ if and only if
\begin{align}
0 \leq&\; Q_{t+1}^{i-1}(y_1)+Q_{t+1}^{i-1}(y_2)\nonumber\\
&\;-Q_{t+1}^{i-1}(\min\{y_1,y_2\})-Q_{t+1}^{i-1}(\max\{y_1,y_2\})\label{eq:submodular}
\end{align}
holds for all $(y_1,y_2)\in Z(x_{t+1}^i)\times Z(x_{t+1}^i)$ (steps 1 and 2). Note that this is not necessarily the case for the approximate value function, even if the exact value function is submodular. We then distinguish between two cases:

\emph{Case I:} If $Q_{t+1}^{i-1}$ is submodular on $Z(x_{t+1^i})$, we locally compute the exact DP stage problem on the set $Y_+(x_t+1)^{i-1}$, i.e.\ $\{\mathcal{T}Q_{t+1}^{i-1}(y)\}_{y \in Y_+(x_{t+1}^i)}$, to construct the hyperplane through $\left\{\left(y,(\mathcal{T}Q_{t+1}^{i-1})(y)\right)\right\}_{y \in Y_+(x_{t+1}^i)}$ (step 3).

\emph{Case II:} If $Q_{t+1}^{i-1}$ is not submodular on $Z(x_{t+1^i})$, we need to compute a submodular upper bound on $Q_{t+1}^{i-1}$, which is readily given by the hyperplanes from which $Q_{t+1}^{i-1}$ is constructed. Therefore, we select the hyperplane $H_{t+1}^{j^*-1}$ that minimises the value at the evaluation point $x_{t}^{i}$, i.e.\
\begin{align}
Q_t^{i}=&\; \min\left\{ \mathcal{T}H_{t+1}^{j^*-1},Q_t^{i-1}\right\}\nonumber, \text{ where }\\
&\; j^* \in \underset{j\in J_{t+1}^{i-1}}{\text{argmin}}\left\{\left(\mathcal{T}H_{t+1}^{j-1}\right)\left(x_{t+1}^i\right)\right\}
\end{align}
and where $J_{t+1}^{i-1}$ is the set of supporting hyperplanes, i.e.\
\begin{equation}
J_{t+1}^{i-1}(x):=\underset{j \in \{0,1,\dots,i-1\}}{\text{argmin}}H_{t+1}^j(x),
\end{equation}
for all $(i,t,x)\in I\times T\times X$ (steps 5 and 6). Therefore, this creates the locally tightest upper bound. Finally, we take the pointwise minimum of the approximate value function at the previous iteration $Q_t^{i-1}$ and the newly created hyperplane $H^*$ to obtain the new value function approximation (step 8).

Similarly to the non-linear stochastic dual DP update from Section \ref{sec:gen_approxalg}-B and in contrast to the affine value function approximation update from Section \ref{sec:gen_approxalg}-A, the approximate value function generated by the gradient-bounded DP update is non-linear in $x$ as it is given by the pointwise minimum of affine functions in $x$. Assuming that the initialiser is an upper bound on the exact value function, which can be satisfied if we choose the fixed point for this purpose (as discussed at the beginning of Section \ref{sec:gen_approxalg}), the approximate value function is an upper bound to th exact value function as shown in \cite[Proposition 1]{LEBEDEVETAL2020A}. Finally, the advantage of gradient-bounded DP over non-linear stochastic dual DP is that only convex optimisation problems need to be solved to compute the update, which makes gradient-bounded DP more resilient against computational stability problems than non-linear stochastic dual DP as shown in Section \ref{sec:casestudy}.

\section{Profit-generation performance criterion}\label{sec:bounds_case}
Since the profits that all three algorithms generate are random variables, we can quantify their performance with probabilistic guarantees by performing validation runs, i.e.\ by simulating customer decisions forward in time and pricing based on the most refined approximate value function. Let the profit that we obtain in each of $k_{\max}$ validation runs be $l_\mathrm{v}(k)$ for all $k\in K:=\{1,\dots,k_{\max}\}$. Let $[l_{-},l_{+}]$ denote the (finite) support of the distribution of $l_{\mathrm{v}}(k)$ for any $k\in K$. In our case, $l_+=V^*(0) = (\bar{d}+r)\innerprod{\mathbf{1}}{\bar{x}}-C(\bar{x})$, where we use the fixed point from \eqref{eq:fixedpoint} and $l_-=-C(0)$. We then compute the empirical mean $\bar{l}_{\mathrm{v}}$ and empirical standard error $\sigma_{\mathrm{v}}$ as 
\begin{subequations}
\begin{align}
\bar{l}_{\mathrm{v}} &:= k_{\max}^{-1}\sum_{k=1}^{k_{\max}} l_{\mathrm{v}}(k),\\
\sigma_{\mathrm{v}}  &:= \sqrt{(k_{\max}-1)^{-1}\sum_{k=1}^{k_{\max}} \left(l_{\mathrm{v}}(k)-\bar{l}_{\mathrm{v}}\right)^2}.
\end{align}
\end{subequations}
For \emph{any} of the three algorithms considered, we can then quantify the performance of a pricing policy using the maximum of the two following bounds, presented in \cite[Proposition 7]{LEBEDEVETAL2020B}, which state the expected profit which can be guaranteed with confidence $(1-\alpha) \in (0,1)$ after observing $k_{\max}$ validation samples. Recall that $\mathbb{E}$ denotes the expectation operator and that Pr$(\cdot)$ denotes the probability of its argument.
\begin{prop}\label{pr:ebounds}
Fix any significance level $\alpha\in (0,1)$. Then $\Pr(\mathbb{E}\bar{l}_{\mathrm{v}}\geq l^*) \geq 1-\alpha$, for all $l^* \in \{l_{\mathrm{B}},l_{\mathrm{D}}\}$, where:
\begin{subequations}\label{eq:bounds}
\begin{align}
l_{\mathrm{B}}&:=\bar{l}_{\mathrm{v}} - \sqrt{\frac{2\sigma_{\mathrm{v}}\ln(\frac{2}{\alpha})}{k_{\max}}}-\frac{7(l_{+}-l_{-})\ln(\frac{2}{\alpha})}{3(k_{\max}-1)} \text{ and }\\
l_{\mathrm{D}}^{\mathbb{E}}&:=\int_{l=0}^{\infty}1-\min\left\{1,F_K(l) + \sqrt{\frac{\ln(\frac{1}{\alpha})}{2k_{\max}}}\right\}\mathrm{d}l,
\end{align}
\end{subequations}
where $F_K$ denotes the empirical cumulative distribution function of $\{l_{\mathrm{v}}(k)\}_{k\in K}$, i.e.\ \mbox{$F_K(l):=k_{\max}^{-1}\sum_{k\in K}\mathbbm{1}(l_{\mathrm{v}}(k)\geq l)$}, where $\mathbbm{1}$ denotes the indicator function.
\end{prop}

The proof can be found in \cite[Appendix A.6]{LEBEDEVETAL2020B}. Strictly speaking, we assume non-positive fixed costs $C(0)$ for the bounds to hold (see \cite[Assumption 5]{LEBEDEVETAL2020B}). Hence, we set $l_- = -C(0)=0$ for our case study. However, since the fixed costs do not impact the pricing policy of any of the algorithms considered, these become irrelevant for our analysis. 

\section{Case Study}\label{sec:casestudy}
In the following three sections, we present the numerical analysis that compares the three value function approximation algorithms stated and analysed in Section \ref{sec:gen_approxalg}. To this end, we generate particular instances of the revenue management problem in attended home delivery presented in Section \ref{sec:problem}. We use the parameter values in \cite{YANG2017} as a base case and modify these parameters to simulate various scenarios and conduct a sensitivity analysis. In Section V-\ref{sec:case_exact}, we analyse the performance of the three algorithms under the assumption that the model parameters are known accurately. Then, we simulate how well the algorithms perform when they are trained on the data in V-\ref{sec:case_exact}, but being tested on scenarios, where the expected demand (Section V-B 1) or customer choice parameters (Section V-B 2) differ from the model.

\subsection{Exact model analysis}\label{sec:case_exact}
In this section, we adapt the numerical case study parameters from \cite{YANG2017} to arrive at the set-up defined in Table \ref{tab:case_exact_par} below.
\begin{table}[b]
\caption{Exact model parameters.}
\begin{center}
\renewcommand{\arraystretch}{\mystretch}
\begin{tabular}{r|l}\label{tab:case_exact_par}
	$S$ & $\{1,2,\dots,17\}$ \\
	$\lambda$ & $0.8$ \\
	$\left[\underline{d},\bar{d}\,\right]$ & $[\textrm{\pounds}0,\textrm{\pounds}10]$ \\
	$\beta_c,\beta_d$ & $-2.5087,-0.0766$\\
	$\{\beta_s\}_{s\in S}$ & $\{-1.0305, -0.3591, 0.3107, 0.5922, 0.6154,$\\
	&$0.0796, 0.5356, -0.2415,-0.6286, -1.6736,$\\
	&$-0.4351, -0.161, 0, 0.2533, 0.0736, 0.562,$\\
	&$0.2346\}$\\
	$r$ & \pounds$34.53$
\end{tabular}
\end{center}
\end{table}
We also use the same step sizes for the affine value function update (see \eqref{eq:avfa_update} in Section \ref{sec:avfa_update}) as in \cite{YANG2017}, i.e.\ $\alpha_1:=0.0001, \alpha_2:=0.00025, \alpha_3:=0.00014$. In addition to these fixed parameters, we consider two more parameters, which we will vary as described further below.

First, we vary the delivery capacity of each time slot by varying the size of the delivery sub-area under consideration to simulate an urban, suburban and rural scenario. Each scenario has a different value of capacity per delivery time slot $\bar{x}$, which influences the variable delivery cost. In practice, the mapping between the characteristics of the delivery-subarea and the delivery capacity for all delivery time slots may depend on a lot of factors like infrastructure, traffic and weather conditions, however for the purpose of our case study, we use a simplified model from \cite{DAZANGO1987,YANG2017}, which derives the delivery slot capacity as follows: Suppose that the delivery sub-area is rectangular and has length $L$ and width $W$ as shown in Fig.\ \ref{fig:subarea}. 

\begin{figure}[H]
\centering
\includegraphics[width=0.4\linewidth]{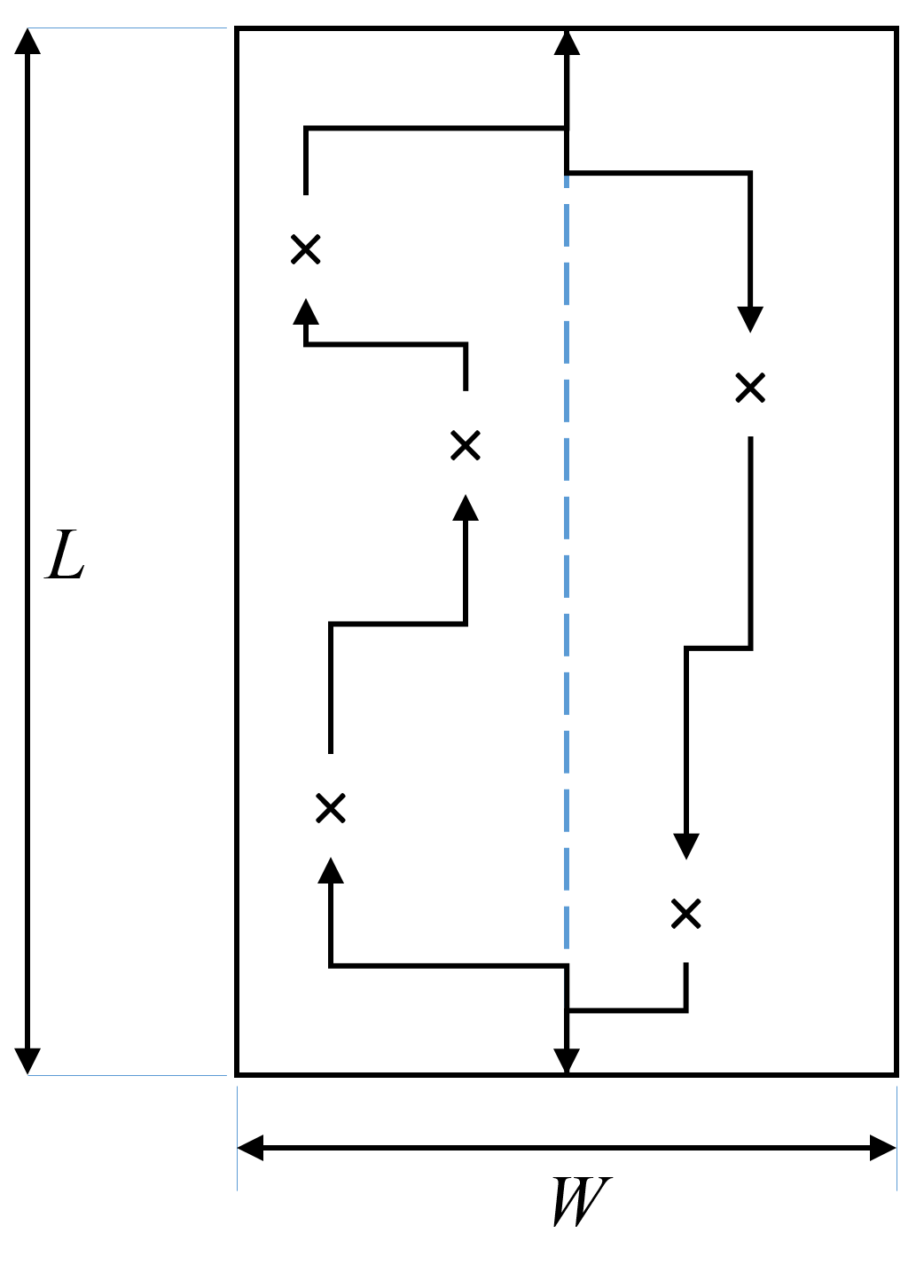}
\caption{Delivery subarea schematic with length $L$ and width $W$. Crosses indicate customer locations served in a particular delivery time slot.}
\label{fig:subarea}
\end{figure}

We further suppose an average delivery truck velocity of $\omega=25$ mph and a cost per mile of $\xi=\pounds0.25$. We assume that in each delivery time slot, the truck travels back and forth along the length $L$ of the delivery sub-area; along the half-width $[0,W/2]$ of the delivery sub-area in one direction and along the other half-width $[W/2,W]$ in the other direction. We then assume that customer locations are random, uniformly distributed in the delivery sub-area and that the truck travels Manhattan distances. This implies that the average distance travelled between two customers along the axis aligned with the width of the sub-area is $1/3$ times the half-width $W/2$. This results in a variable delivery cost of
\begin{equation}
c_{\mathrm{var}}:= \xi \times W/6,
\end{equation}
as shown in \cite{DAZANGO1987}. We find $W$ from the condition that, in any delivery slot, the delivery truck must be able to make $\bar{x}$ deliveries and an additional assumption that $L=2W$. The last choice is arbitrary and our results do not change qualitatively for other ratios between $L$ and $W$. The total travelling distance in every delivery time slot thus becomes
\begin{align}
\omega \times 1\mathrm{h} &= 2L + \bar{x}\frac{W}{6}\nonumber \\
\Rightarrow W &= \frac{\omega}{4+\bar{x}/6}.
\end{align}
This finally implies that 
\begin{equation}
c_{\mathrm{var}}:= \frac{\xi\omega}{24+\bar{x}}.
\end{equation}

Second, we vary the expected demand, i.e.\ the expected number of customer arrivals on the booking website. This quantity is given by $\lambda\bar{t}$. Since it is reasonable to keep $\lambda\approx0.8$ for customer choice parameter estimation purposes (see \cite{YANG2016}), we fix $\lambda=0.8$ and vary $\bar{t}$ to achieve a total demand level corresponding to $\phi n \bar{x}$, where $n\bar{x}$ is the total delivery capacity for all slots and $\phi\in \Re$ is a demand factor, such that $\phi \in \Phi:=\{1/8,1/4,1/2,1,2,4,8\}$. Hence, we find $\bar{t}$ as
\begin{equation}
\bar{t} \approx \frac{\phi n \bar{x}}{\lambda},
\end{equation} 
for all $\phi \in \Phi$ and where the approximation comes from rounding $\bar{t}$ to the nearest integer. For all these scenarios, we compute the profit that is reached in expectation with confidence 99\%, by computing 100 validation samples for each scenario and each algorithm and using the tighter of the two bounds from Section \ref{sec:bounds_case}.

In general, we observe that the non-linear stochastic dual DP algorithm produces higher expected profits than the affine value function approximation algorithm, while taking more time to compute a good solution. However, the gradient-bounded DP algorithm exhibits the strengths of both other algorithms: very similar profit generation performance to non-linear stochastic dual DP and similar speed to the affine value function approximation algorithm. For example, Fig.\ \ref{fig:Exp01_time} below shows the computation time that it takes for the three algorithms to reach at least 95\% of their maximum expected profit with 99\% confidence for various demand factors and delivery time slot capacities. Non-linear stochastic dual DP always takes longest to compute out of the three algorithms. Computation time also tends to increase for non-linear stochastic dual DP as demand factor or slot capacity increase. For capacity 20, it takes about 4 times longer to compute the solution for demand factor 8, compared with the other two algorithms. This time factor increases to about 10 as we decease the demand factor to $1/8$. 

\begin{figure}[b!]
\centering
\includegraphics[width=\mywidth\linewidth]{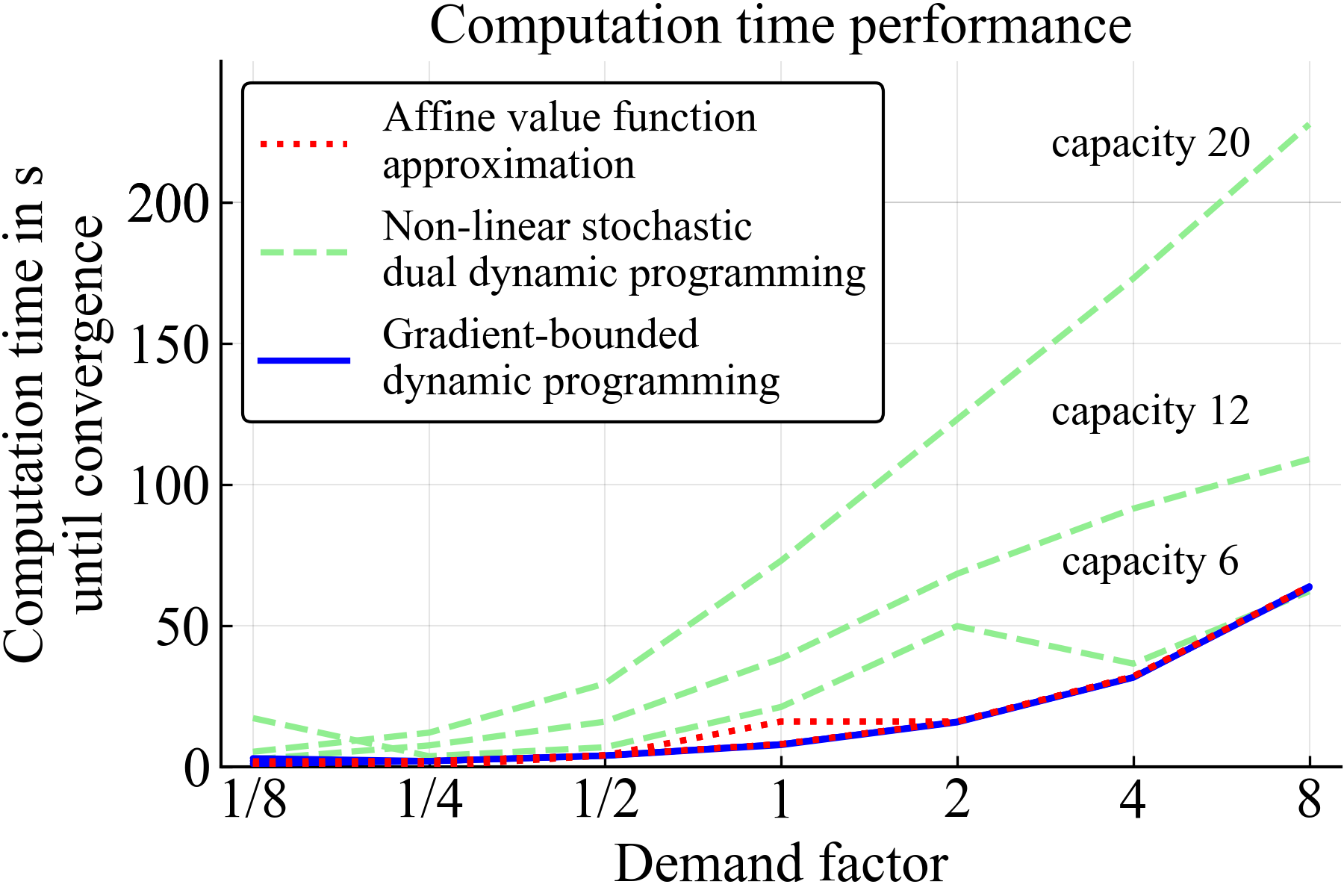}
\caption{Computational time to reach 95\% of the maximum expected profit with 99\% confidence for each of the three algorithms against demand factor and delivery slot capacities.}
\label{fig:Exp01_time}
\end{figure}

Affine value function approximation and gradient-bounded DP take similar time to converge to their respective optimal solutions. Computation time does not vary across slot capacities for these algorithms for all but one scenario: For demand factor 1 and slot capacity 6, affine value function approximation takes twice as long to converge compared with gradient-bounded DP. One possible explanation for this is that for this particular scenario, it might be computationally involved to find the optimal affine value function approximation since for a medium demand factor it is difficult to find a single affine value function approximation that works well for all sample paths. Some slots might sell out, some might not, which increases the need for a more flexible solution that gradient-bounded DP can provide.

Another issue observed is that the non-linear stochastic dual DP algorithm becomes computationally unstable under certain conditions. For example, for demand factor 8 and slot capacity 12, its profit generation performance decreases over time as can be seen in Fig.\ \ref{fig:Exp01_degradation} below. This might appear counter-intuitive at first, but is in line with our theoretical analysis from Section \ref{sec:gen_approxalg}-B: We conjecture that this is due to the difficulty of finding global maxima of non-convex optimisation problems. If the algorithm converges to a local maximum, the value function approximation is no longer guaranteed to be an upper bound on the exact value function (see Proposition \ref{pr:lagrange}). Over time, this then leads to a compounding of errors caused by suboptimality, i.e.\ instead of increasing, the expected profit decreases as more cuts are added to the approximate value function. A practical way to circumvent this problem is to compute the expected profit with 99\% confidence after each iteration and to pick the iteration which produces the maximum expected profit with 99\% confidence. In the example of Fig.\ \ref{fig:Exp01_degradation}, the best solution is found after the first iteration -- the optimal policy is dominated by pricing all slots at the maximum charge $\bar{d}$ for all time steps, since the high demand factor 8 almost guarantees that all slots will be sold out for any choice of admittable prices. Hence, over time invalid cuts accumulate, which results in a degradation of the profit performance.

\begin{figure}[b!]
\centering
\includegraphics[width=\mywidth\linewidth]{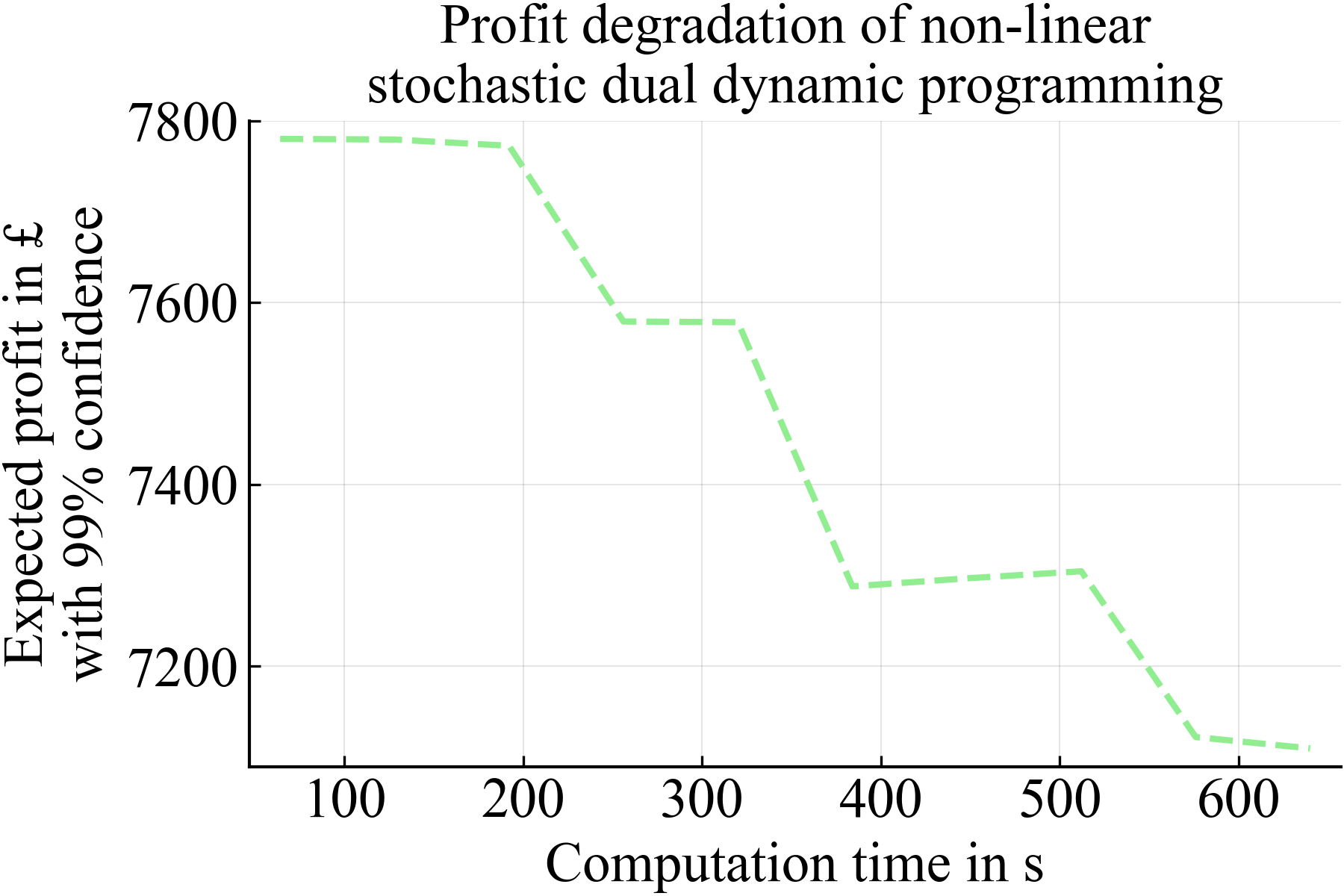}
\caption{Expected profits of non-linear stochastic dual DP for demand factor 8 and slot capacity 12 decrease as more iterations are added over time.}
\label{fig:Exp01_degradation}
\end{figure}

Comparing the expected profits obtained between the three algorithms, we observe that gradient-bounded DP always generates the highest expected profit with 99\% confidence or is within 1\% of the optimal value, when the demand factor is so high that demand saturates and all three algorithms perform very similarly. This saturation behaviour can be seen in Fig.\ \ref{fig:Exp01_profit} below, where we also show that for demand factors 1 and lower, gradient-bounded DP produces between 10 and 15\% more expected profit with 99\% confidence than affine value function approximation. At the same time, gradient-bounded DP performs similarly to non-linear stochastic dual DP in most scenarios. However, gradient-bounded DP generates up to 10\% more profit than non-linear stochastic dual DP for small demand factor $1/8$ and capacity 6 as well as for large demand factor 8 across all slot capacities.

Overall, we conclude that gradient-bounded DP performs best in this exact model experiment, because it outperforms affine value function approximation in terms of profit generation while being similarly fast and at the same time, gradient-bounded DP is more than four times faster than non-linear stochastic dual DP while generating very similar profit.

\begin{figure}[b!]
	\centering
	\includegraphics[width=\mywidth\linewidth]{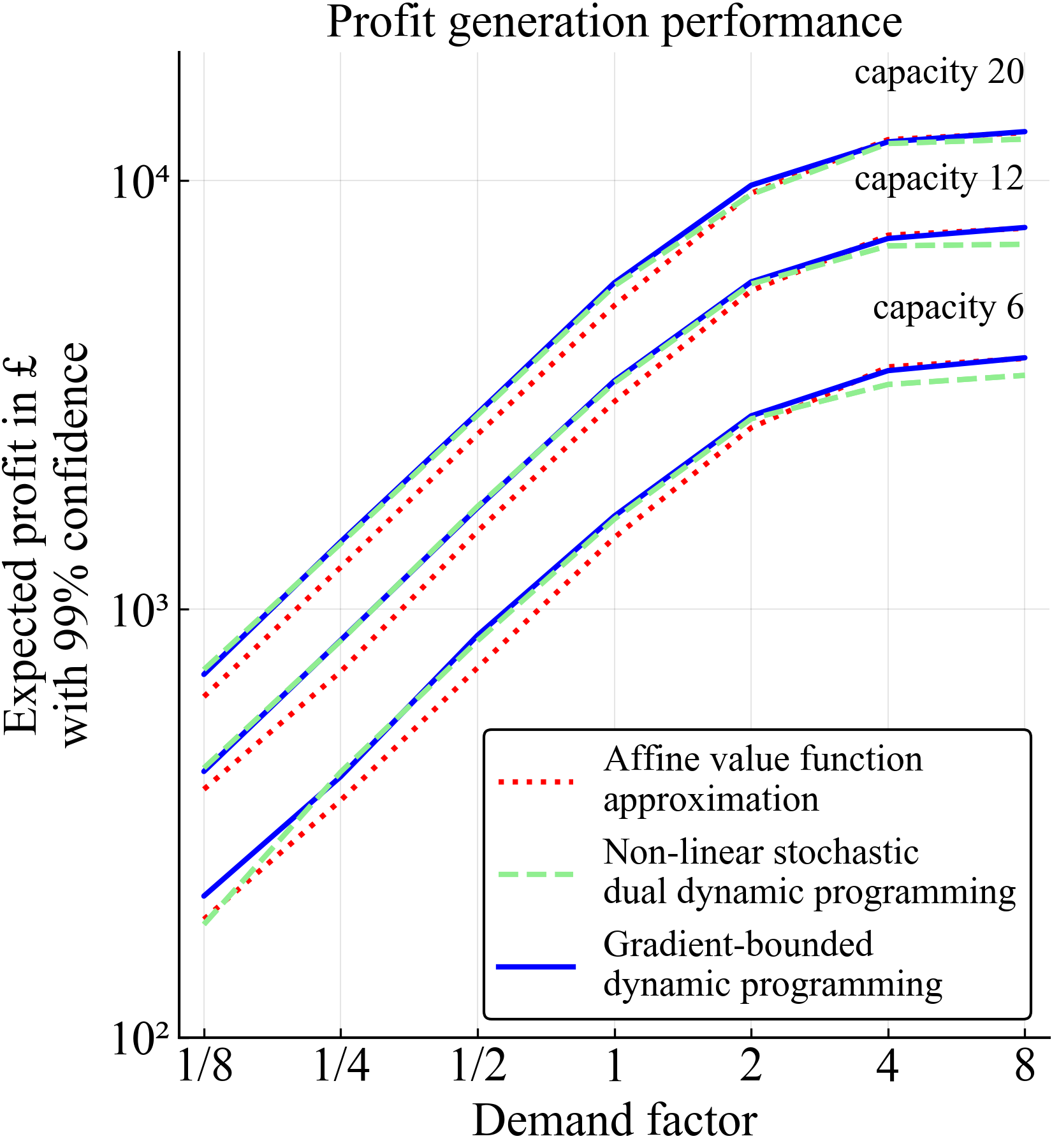}
	\caption{Expected profits with 99\% confidence of the three algorithms against demand factor and for all delivery slot capacities.}
	\label{fig:Exp01_profit}
\end{figure}

\subsection{Parameter sensitivity analysis}
We assume in the previous section that the parameters, in particular the customer arrival rate $\lambda$ and the customer choice model parameters $\beta_c, \beta_d$ and $\{\beta_s\}_{s\in S}$, are known exactly, which is not the case in practice. Hence, we now investigate how well the pricing policies obtained by the three algorithms in the previous section perform on perturbed models.

\subsubsection{Uncertain demand analysis}\label{sec:case_demand}
Consider the case that we derive value function approximations on the assumption that $\lambda=0.8$, but actually, this value differs, i.e.\ more or less customers arrive on the booking website than anticipated. We model one scenario where demand is 25\% lower, hence $\lambda=0.6$, and a second scenario where demand is 25\% higher, hence $\lambda=1$. The resulting behaviour is very similar for all algorithms, as can be seen in Fig.\ \ref{fig:Exp03}.

\begin{figure}[ht!]
\centering
\begin{subfigure}{\mywidth\linewidth}
\centering
\includegraphics[width=\textwidth]{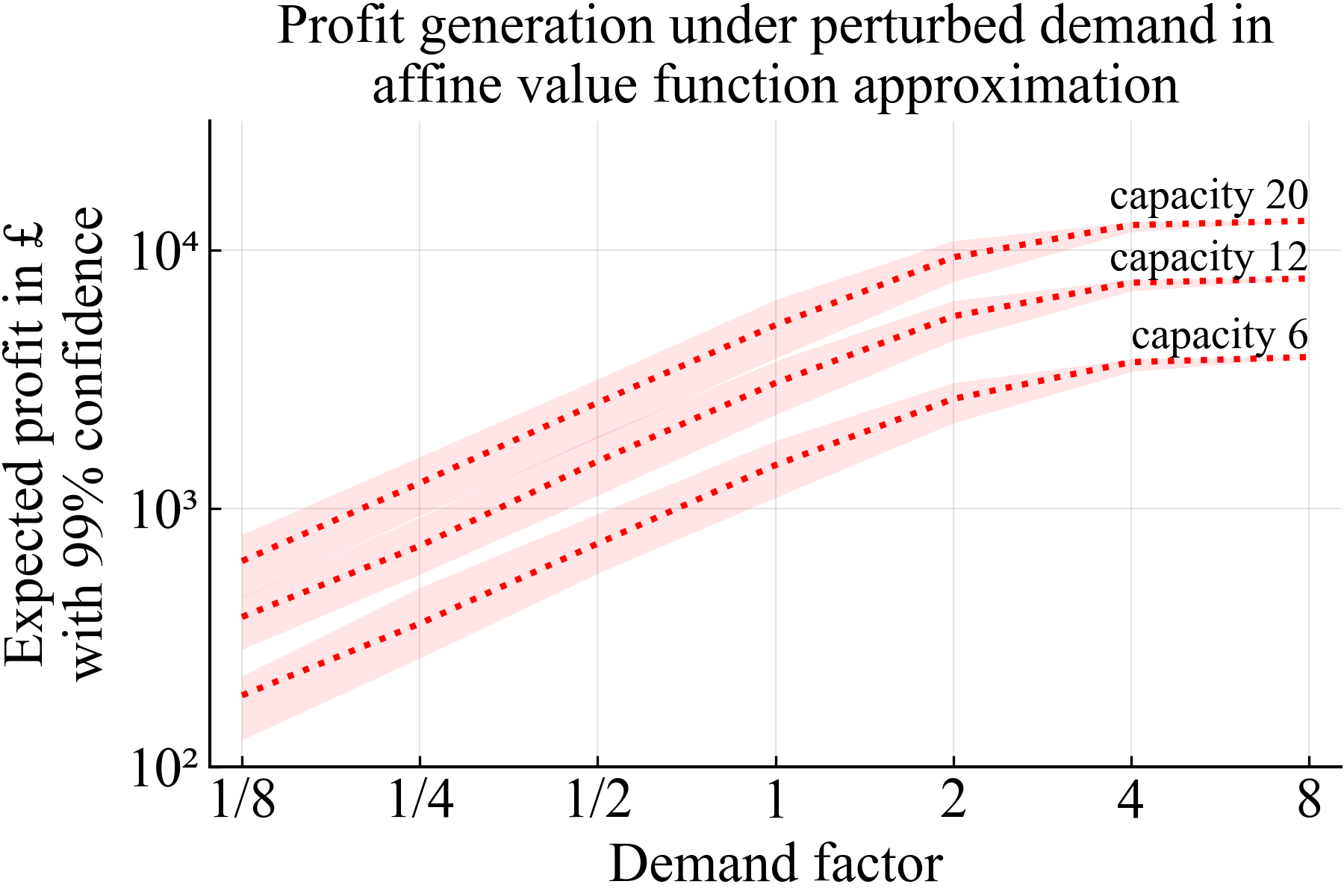}
\caption{}
\end{subfigure}\\
\vspace{8mm}
\begin{subfigure}{\mywidth\linewidth}
\centering
\includegraphics[width=\textwidth]{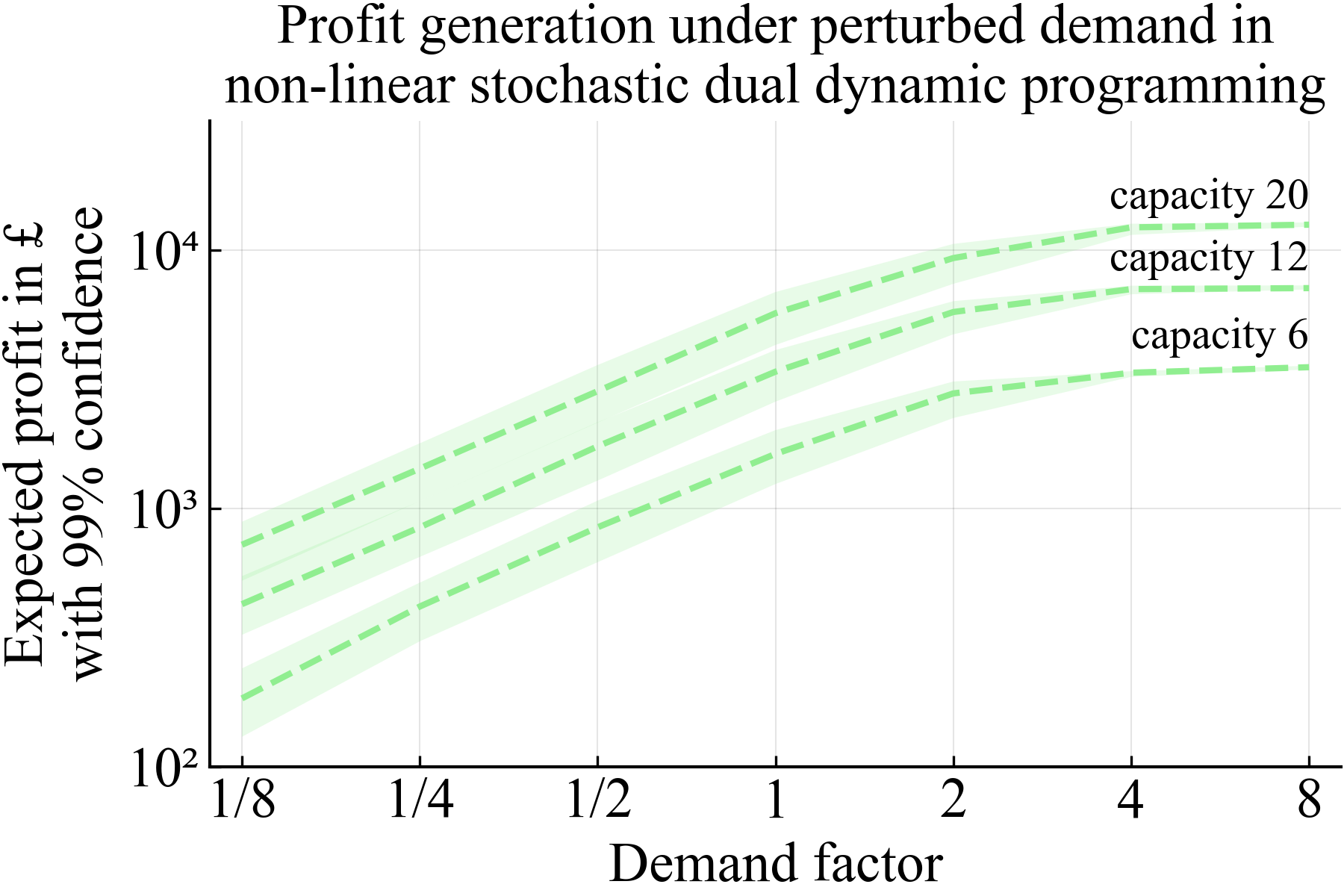}
\caption{}
\end{subfigure}\\
\vspace{8mm}
\begin{subfigure}{\mywidth\linewidth}
\centering
\includegraphics[width=\textwidth]{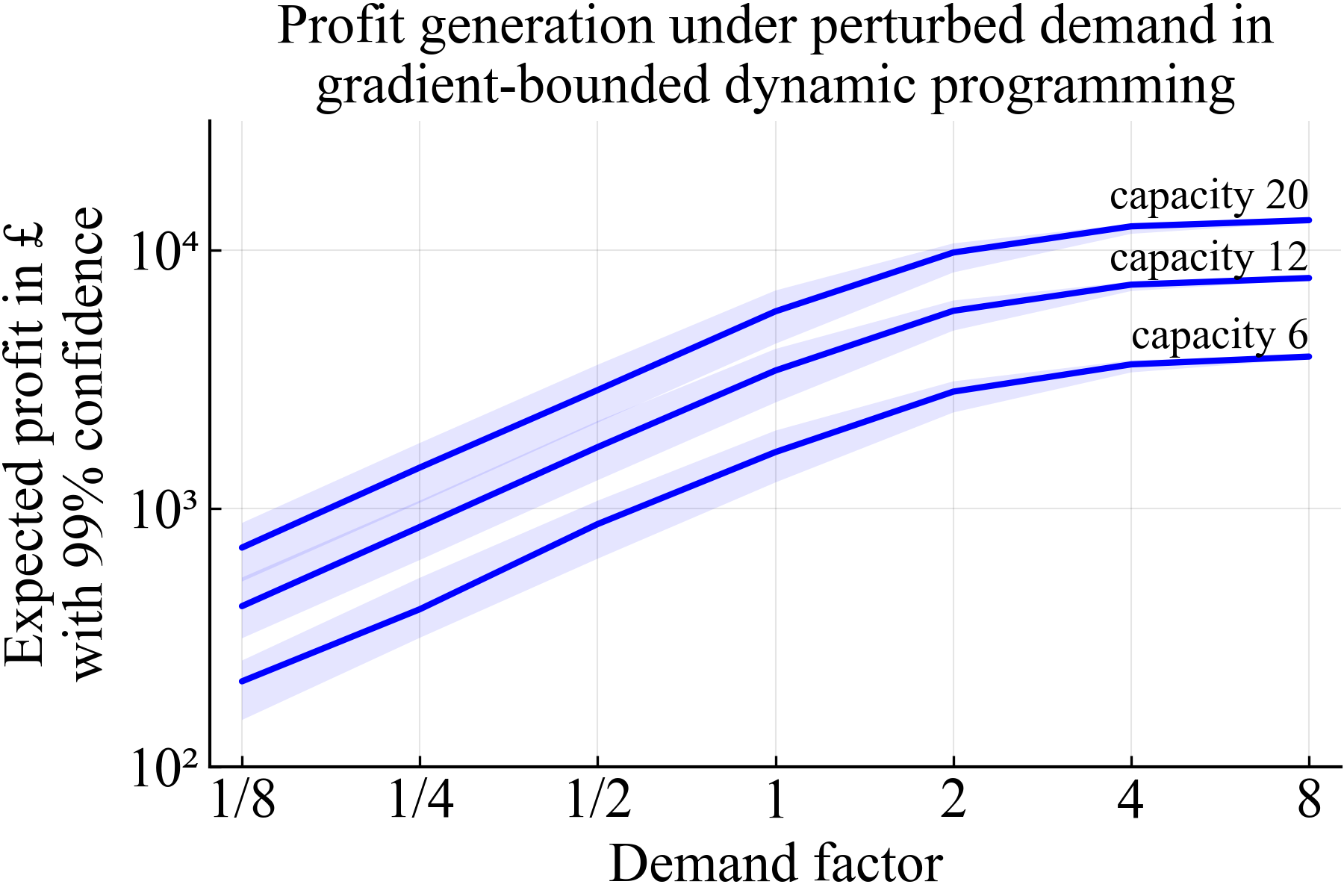}
\caption{}
\end{subfigure}
\caption{Expected profits with 99\% confidence for affine value function approximation (a), non-linear stochastic dual DP (b) and gradient-bounded DP (c) under perturbed customer arrival rate. Shaded regions indicate how expected profit with 99\% confidence increases when the customer arrival rate increases from 0.6 to 1. The lines correspond to an arrival rate of 0.8.}
\label{fig:Exp03}
\end{figure}
The numerical results are almost identical to the ones obtained in the exact model analysis experiment in Section V-\ref{sec:case_exact}. The relative performance in terms of profit generation and computational time across all three algorithms is very similar. The only difference to the results from Section V-\ref{sec:case_exact} are the absolute profit levels obtained for demand factors 2 and smaller, which scale proportionally with the customer arrival rate $\lambda$, as one would expect intuitively. For larger demand factors, the expected profit saturates across all $\lambda \in \{0.6,0.8,1\}$. This is because almost all slots are expected to sell out at the maximum delivery charge before the end of the time horizon. In conclusion, all three algorithms are equally capable of compensating for variations in the customer arrival rate and their relative performance remains unchanged in comparison with the exact model experiment in Section V-\ref{sec:case_exact}.

\subsubsection{Uncertain customer choice parameters}\label{sec:case_choice}
Consider the case that customer preferences across slots and prices are misspecified. To model this, we now corrupt the parameter estimates $\beta_c, \beta_d$ and $\{\beta_s\}_{s \in S\cup\{0\}}$ by additive Gaussian noise. This choice of distribution is justified because, in the limit as the number of data points used for estimating the customer choice parameters tends to infinity, the error between estimated and true customer choice parameter value vector is a Gaussian with zero mean \cite[Chapter 8.6]{TRAIN2009}. 

We consider three scenarios in which we vary the level of estimation error by setting the variance of the Gaussian to $\sigma^2 \in \{0.01,0.1,1\}$. With these three noise levels, we sample the sets of customer choice parameters, which we hold fixed for all validation runs in this experiment. Note that we do not have to worry about normalising the probability distribution, since the multinomial choice model is normalised for all possible parameter values. The numerical values used in our analysis are documented in Table \ref{tab:corrupted_beta} in Appendix \ref{ap:tables}. In practice, we could estimate the value of $\sigma^2$ from the data to test the robustness of candidate value function approximation algorithms with respect to uncertainty in the customer choice model. We document how the profit generation performance of the three algorithms degrades in comparison with the ideal scenario in the previous section in Fig.\ \ref{fig:Exp02} below.

\begin{figure}[ht!]
\centering
\begin{subfigure}{\mywidth\linewidth}
\centering
\includegraphics[width=\textwidth]{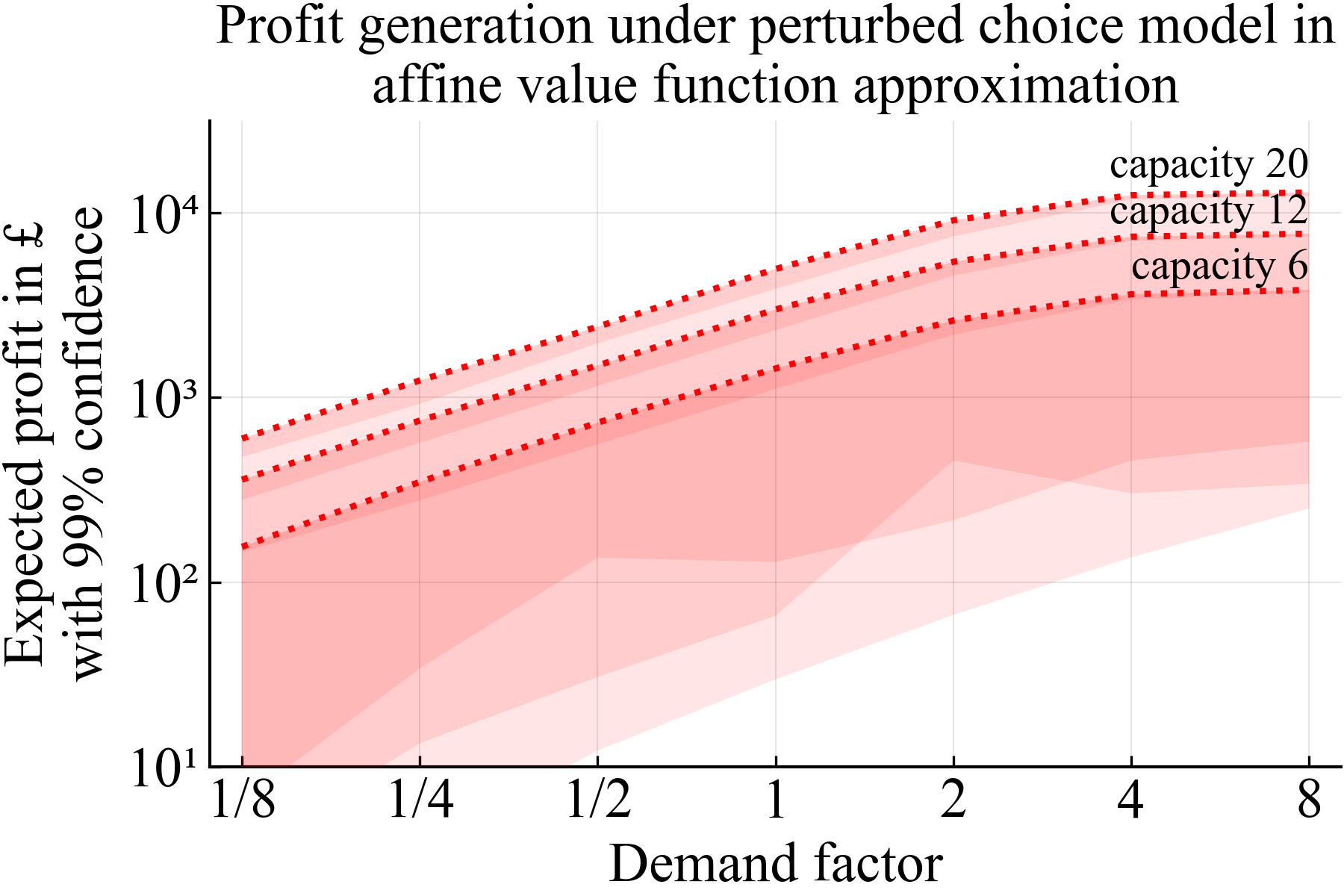}
\caption{}
\end{subfigure}\\
\vspace{8mm}
\begin{subfigure}{\mywidth\linewidth}
\centering
\includegraphics[width=\textwidth]{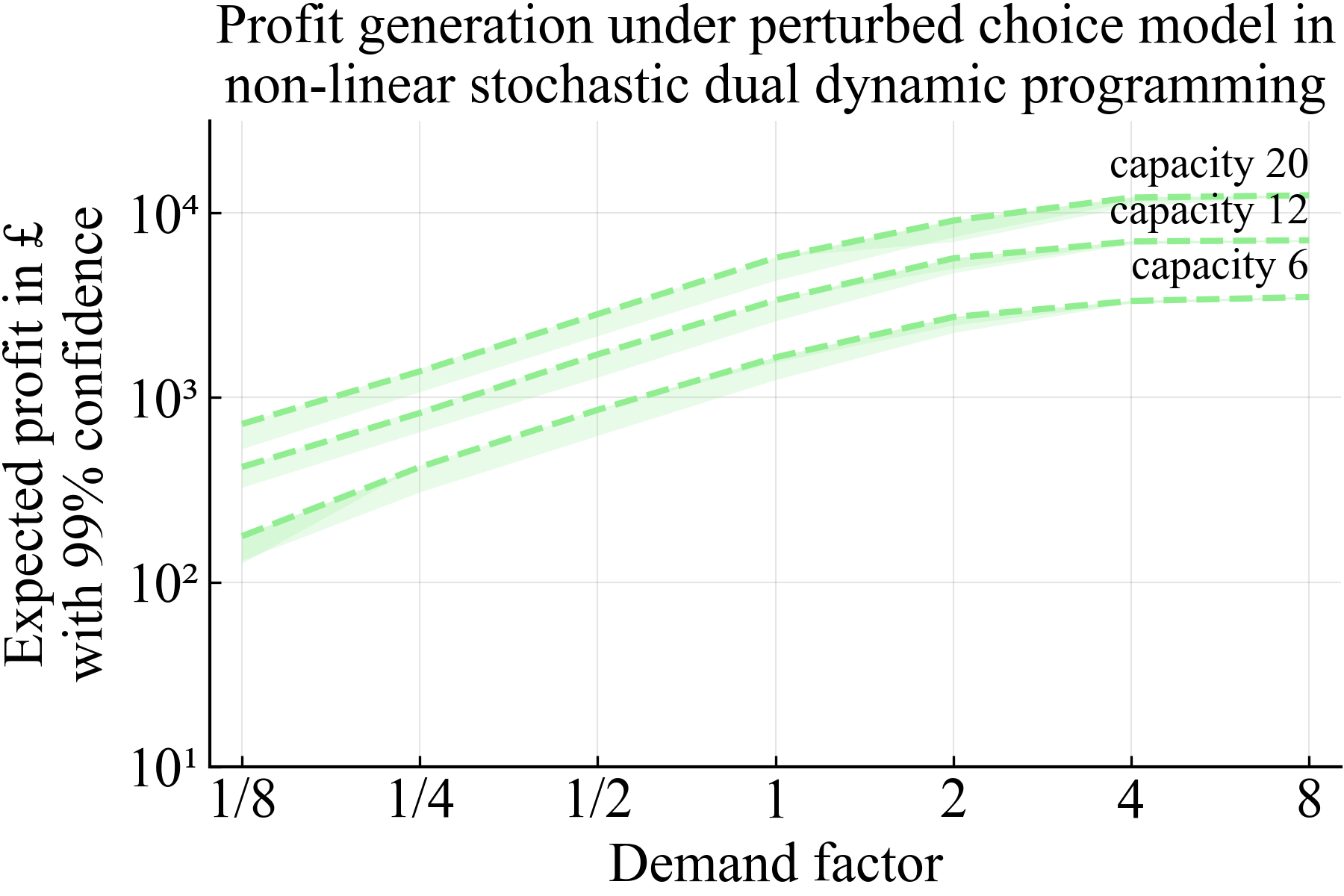}
\caption{}
\end{subfigure}\\
\vspace{8mm}
\begin{subfigure}{\mywidth\linewidth}
\centering
\includegraphics[width=\textwidth]{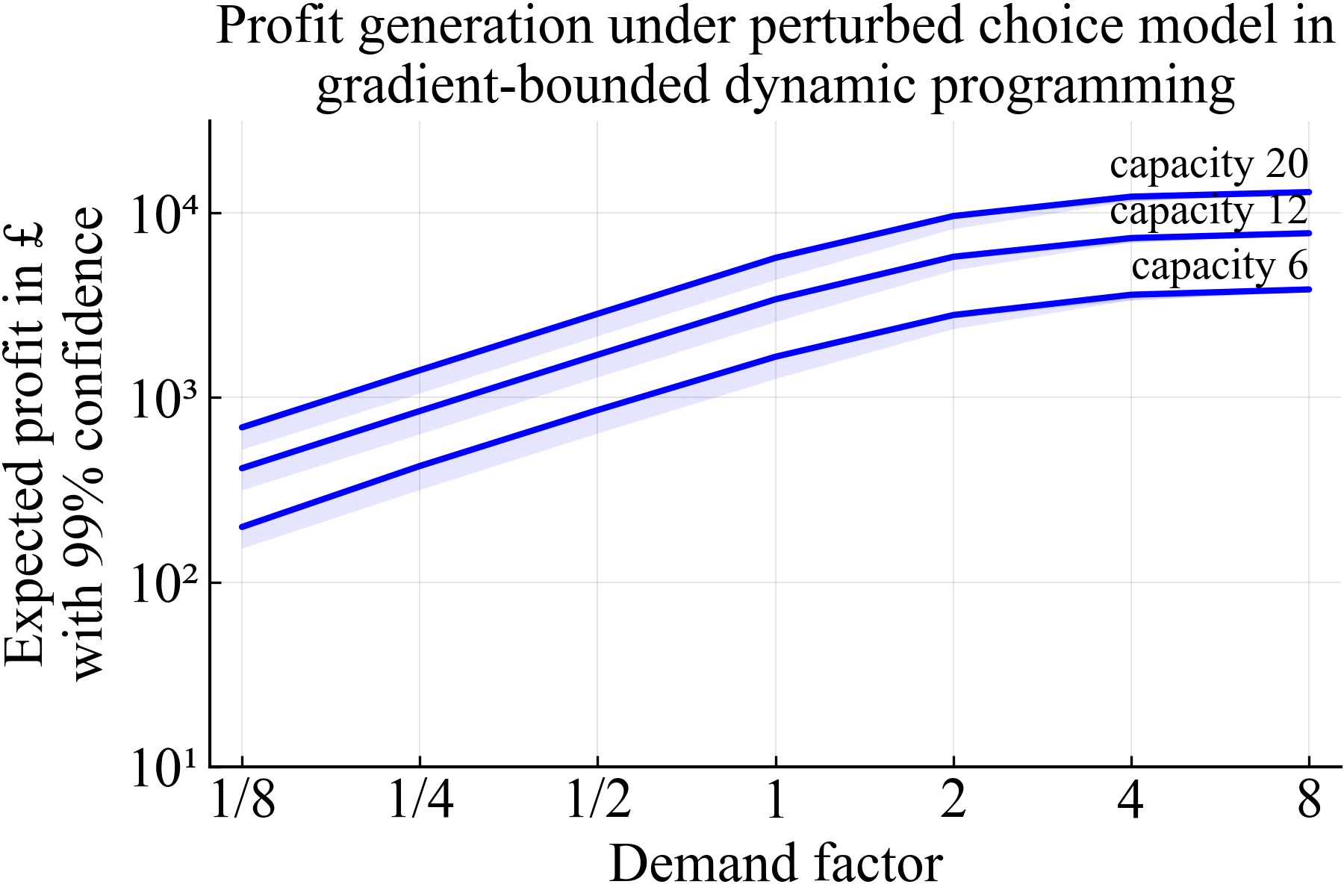}
\caption{}
\end{subfigure}
\caption{Expected profits with 99\% confidence for affine value function approximation (a), non-linear stochastic dual DP (b) and gradient-bounded DP (c) under perturbed customer choice model parameters from Table \ref{tab:corrupted_beta}. Lines indicate performance at $\sigma^2 =0.01$ and shaded regions indicate performance as $\sigma^2$ increases to 0.1 and 1.}
\label{fig:Exp02}
\end{figure}

As we see in Fig.\ \ref{fig:Exp02}(b) and (c), non-linear stochastic dual DP and gradient-bounded DP are both very robust against uncertainty in the customer choice model. Only for $\sigma^2 = 1$, there is a substantial degradations in profit generation performance. In contrast, Fig.\ \ref{fig:Exp02}(a) shows that even small uncertainties in the customer choice model have substantial negative impact on the profit generation performance of the affine value function approximation algorithm, decreasing expected profit with 99\% confidence by about an order of magnitude for $\sigma^2 = 1$. We conjecture that this is due to the lack of state feedback in the affine value function approximation solution as detailed in Section \ref{sec:gen_approxalg}-A: For any $t\in T$, the suggested optimal slot price vector is identical for all $x$ strictly inside the set of feasible states $X$, because the affine value function approximation has constant gradient for all these points. Since the other two algorithms both generate a piecewise affine approximate value function, gradients and hence optimal delivery prices vary depending on the particular state-time pair $(x,t)\in X\times T$.

Overall, we conclude that both gradient-bounded DP and non-linear stochastic dual DP  increase their relative profit-generation advantage over affine value function approximation when the parameter estimates of the customer choice model are not known exactly.




\section{Conclusions and future work}\label{sec:conclusions}
In this paper, we analysed -- theoretically and numerically -- three approximate dynamic programming algorithms to find approximately optimal delivery slot prices in the revenue management problem in attended home delivery. From a control-theretical perspective, we identified limitations in the affine value function approximation algorithm and the non-linear stochastic dual dynamic programming algorithm. Through our numerical analysis, we showed how gradient-bounded dynamic programming can overcome these limitations. In our case study, we compared the performance of all three algorithms, i.e.\ profit-generation capabilities and computational time, in a number of scenarios. Overall, our numerical analysis shows that the gradient-bounded dynamic programming algorithm exhibits superior performance, since the affine value function approximation algorithm cannot reach its profit-generation capabilities and since the non-linear stochastic dual dynamic programming algorithm cannot reach its computational speed and computational stability properties.

Possible directions for future work include investigating the numerical performance of these algorithms for other network revenue management problems and extending the promising gradient-bounded dynamic programming approach to other customer decision models than multinomial logit.

\bibliographystyle{plain}
{\small\bibliography{CaseStudy_Ref01}}
\begin{appendices}
\section{Problem reformulation for non-linear stochastic dual dynamic programming}\label{ap:biconcave_reform}
In this section, we show that we can re-write \eqref{eq:vstar}, a non-convex optimisation problem, such that the resulting objective function is biconcave in two newly introduced variables as follows.

Equation \eqref{eq:vstar} depends on $z$, which is defined to be integer-valued. Since $Q_{t+1}^{i-1}$ is the pointwise minimum of a finite number of affine functions in $z$, we can make the objective function in \eqref{eq:vstar} concave in a new variable $y\in \Re^n$, which we restrict to be in the convex hull of the original state-space conv$(X)$, i.e.\

\begin{align}
& v^*=\underset{\mu\in \mathcal{M}}{\min} \; \underset{\substack{{d\in D,}\\{y\in \mathrm{conv}(X)}}}{\max}\left\{\lambda\sum_{s\in S}P_s(d)\left(r+d_s+Q_{t+1}^{i-1}(y+1_s)\right)\right.\nonumber\\
& \hphantom{v^*:=}\;\left.+\left(1-\lambda\sum_{s\in S}P_s(d)\right)Q_{t+1}^{i-1}(y)+\innerprod{\mu}{x_{t+1}^i-y}\right\}.
\end{align}

This does not change the optimal solution since optimality can only be reached when $y=x_{t+1}^i$, such that $\innerprod{\mu}{x_{t+1}^i-y}=0$ for all $\mu \in \mathcal{M}$. Moreover, by a similar argument $z=x_{t+1}^i$ for optimality in \eqref{eq:vstar}. Hence, introducing $y$ does not change the overall solution.
Furthermore, as shown in \cite{DONGETAL2009,LEBEDEVETAL2019B}, the maximisation over $d$ can be expressed as a maximisation over transition probabilities, denoted by $p\in [0,1)^n$, by inverting the function $P_s(d)$ for all $(x,s,d) \in X \times S \times D$. Writing $P_s(d)$ now as the elements $p_s$ of the variable $p$ and replacing $d$ by the inversion of $P_s(d)$ as a function of $p_s$ for all $s\in S$, results in an objective function that is concave in $p$, i.e.\
\begin{align}\label{eq:biconcave}
v^*=\;&\underset{\mu\in \mathcal{M}}{\min} \; \underset{\substack{{p\in P,}\\{y\in \mathrm{conv}(X)}}}{\max}\left\{\lambda\sum_{s\in S}p_s\left(r\vphantom{\frac{1}{1}}\right.\right.\nonumber\\
\hphantom{v^*:=}\;&\left.\left. +\;\frac{1}{\beta_d} \left(\ln\left(\frac{p_s}{p_0}\right) -\beta_c-\beta_s)\right)+ Q_{t+1}^{i-1}(y+1_s)\right)\right.\nonumber\\
\hphantom{v^*:=}\;&\left.\left.+\left(1-\lambda\sum_{s\in S}p_s\right)Q_{t+1}^{i-1}(y)+\innerprod{\mu}{x_{t+1}^i-y}\right\}\right.,
\end{align}
where $P$ is formed by imposing the constraints that $p_s \geq 0$, for all $s\in S\cup\{0\}$, $\sum_{s\in S\cup\{0\}} p_s=1$ and $p_s \leq p_0\exp(\beta_c+\beta_s+\beta_d\underline{d})$, for all $s\in S$. One more constraint is given by $p_s \geq p_0\exp(\beta_c+\beta_s+\beta_d\bar{d})$, for all $s\in S$, if $x_s < \bar{x}_s$, for all $(x,s)\in X\times S$. We can express this constraint together with the non-negativity constraint on $p_s$ as $p_s\geq p_0\exp(\beta_c+\beta_s+\beta_d\bar{d}) - \exp(\alpha(x_s-\bar{x}))$, for all $s\in S$, where $\alpha > 0$ is sufficiently large to make the exponential term in the constraint negligible for $x_s < \bar{x}$, e.g.\ we use $\alpha=10$ for our numerical studies.

Notice that for any fixed $p\in P$, the objective function is concave in $y$, because $y$ only appears in $Q_{t+1}^{i-1}$, which is concave in $y$ and in the affine term $\innerprod{\mu}{x_{t+1}^i-y}$. Similarly to \cite[Theorem 1]{DONGETAL2009} and \cite[Lemma 3(i)]{LEBEDEVETAL2019B}, it can be shown that for any $y\in \mathrm{conv}(X)$, the objective function is concave in $p$. It follows that the resulting objective function is biconcave in $(p,y)$, but has non-linear constraints. 

In general, this is a difficult problem. However, since the optimal solution is at $y=x_{t+1}^i$, we can fix $y$ to this value and find the optimal value of $p$ in the maximisation in \eqref{eq:biconcave}. We then use these values for $y$ and $p$ to initialise a non-linear solver, which produces quite reliable results as shown in the numerical example in Section \ref{sec:casestudy}.

\section{Tables}\label{ap:tables}
\begin{table}[b]
\begin{center}
\renewcommand{\arraystretch}{\mystretch}
\caption{Estimation error-corrupted customer choice parameters.}
\begin{tabular}{l | l | l | p{4cm}}\label{tab:corrupted_beta}
$\sigma^2$ & $\beta_c$ & $\beta_d$ & $\{\beta_s\}_{s\in S}$ \\
\hline
0.01 & -2.5143 & -0.0821 & $\{-1.0232, -0.3546, 0.3127, $\\
& & & $0.5860, 0.6211, 0.1056, $\\
& & & $0.5179, -0.2563, -0.6204, $\\
& & & $-1.6762, -0.4313, -0.1621, $\\
& & & $-0.0037, 0.2453, 0.0732, $\\
& & & $0.5590, 0.2192\}$\\
0.1 & -2.5078 & -0.1337 & $\{-1.0866, -0.3378, 0.3152, $\\
& & & $0.6729, 0.6802, 0.1164, $\\
& & & $0.5365, -0.3083, -0.4778, $\\
& & & $-1.7180, -0.3431, -0.0674, $\\
& & & $-0.0291, 0.1821, 0.0411, $\\
& & & $0.6579, 0.2281\}$\\
1 & -2.6071 & -1.2039 & $\{-0.2242, -0.3338, -0.6643, $\\
& & & $0.9198, -0.1940, -1.068, $\\
& & & $-0.5910, 0.9239, -1.0473, $\\
& & & $-1.8225, -0.6845, 1.8887, $\\
& & & $0.7644, -1.4806, 0.4730, $\\
& & & $0.2706, 1.1856\}$
\end{tabular}
\end{center}
\end{table}
\end{appendices}
\end{document}